\documentclass{amsart}
\usepackage{amsmath,amsthm,amsfonts,amssymb,mathtools,mathrsfs,url}
  \usepackage{paralist}
  \usepackage{graphics}
    \usepackage{epsfig}
\usepackage{graphicx}
\usepackage[colorlinks=true]{hyperref}
\usepackage{graphicx}
\usepackage{amsmath}
\usepackage{amssymb}
\usepackage[utf8]{inputenc}
\usepackage[french, main=english]{babel}
\usepackage{subcaption}
\usepackage{enumitem}
\usepackage{url}
\usepackage{xcolor}
\usepackage[toc,page]{appendix}
\usepackage{stmaryrd}
\usepackage{mathrsfs}
\usepackage{mathtools}
\usepackage[abbrev]{amsrefs}
\usepackage{tikz}
\usetikzlibrary{patterns.meta}
\usepackage{pgfplots}
\usepackage{lmodern}

\pgfplotsset{every axis/.append style={
                    axis x line=middle,    
                    axis y line=middle,    
                    axis line style={->}, 
                    xlabel={$x$},          
                    ylabel={$v$},          
            }}

\numberwithin{equation}{section}

\hypersetup{urlcolor=blue, citecolor=green}

\mathtoolsset{showonlyrefs=true}

\calclayout

\newcommand{\R}{\mathbb R} 
\newcommand{\N}{\mathbb N}

\newcommand{\C}{\mathbb C}

\newcommand{\Pro}{\mathbb P}

\newcommand{\1}{\mathbf 1}
\newcommand{\D}{\mathrm{d}}
\newcommand{\om}{\omega}

\newcommand{\e}{\mathrm{e}}

\newcommand{\hop}{\vskip.3cm\noindent} 
\newcommand{\hip}{\vskip.1cm\noindent}

\newtheorem{thm}{Theorem}[section]   
\newtheorem{cor}[thm]{Corollary}
\newtheorem{prop}[thm]{Proposition}
\newtheorem{lem}[thm]{Lemma}
\newtheorem{defi}[thm]{Definition}
\newtheorem{rema}[thm]{Remark}
\newtheorem{hypo}[thm]{Hypothesis}

\DeclareRobustCommand{\SkipTocEntry}[5]{}

\title[Metastability for the Linear Relaxation Boltzmann equation]{Spectral asymptotics and Metastability for the Linear Relaxation Boltzmann equation} 
\author[T. Normand]{Thomas Normand}
\email{thomas.normand@univ-nantes.fr}
\date{}

\begin{document}

\begin{abstract}
We consider the linear relaxation Boltzmann equation in a semiclassical framework.
We construct a family of sharp quasimodes for the associated operator which yields sharp spectral asymptotics for its small spectrum in the low temperature regime.
We deduce some information on the long time behavior of the solutions with a sharp estimate on the return to equilibrium as well as a quantitative metastability result.
The main novelty is that the collision operator is a pseudo-differential operator in the critical class $S^{1/2}$ and that its action on the gaussian quasimodes yields a superposition of exponentials.
\end{abstract}

\maketitle

\tableofcontents

\section{Introduction}
\subsection{Motivations}
We are interested in the linear Boltzmann equation:
\begin{equation}\label{Boltl1}
\left\{
\begin{aligned}
&h\partial_tu+v\cdot h\partial_xu-\partial_xV\cdot h\partial_vu+Q_{\mathcal H}(h,u)=0 \\ 
&u_{|t=0}=u_0
\end{aligned}
\right.
\end{equation}
in a semiclassical framework (i.e in the limit $h \to 0$), where $h$ is a \emph{semiclassical parameter} and corresponds to the temperature of the system.
Here we denoted for shortness $\partial_x$ and $\partial_v$ the partial gradients with respect to $x$ and $v$.
This equation is used to model the evolution of a system of charged particles in a gas on which acts an electrical force associated to the real valued potential $V$ that only depends on the space variable $x$. 
The operator $Q_{\mathcal H}$ is called \emph{collision operator} and models the interactions between the particles.
Here the unknown is the function $u:\R_+\to L^1(\R^{2d})$ giving the probability density of the system of particles at time $t \in \R_+$, position $x \in \R^d$ and velocity $v\in \R^d$.
For our purpose, we introduce the square roots of the usual Maxwellian distributions 
\begin{align}\label{muh}
\mu_h(v)=\frac{\e^{-\frac{v^2}{4h}}}{(2\pi h)^{d/4}} \qquad \text{and } \qquad \mathcal M_h=\e^{-\frac{V}{2h}}\mu_h.
\end{align}
This paper is devoted to the study of the linear BGK model for which the collision operator is
\begin{align}\label{QH}
Q_{\mathcal H}(h,u)=h\Big(u-\int_{v'\in \R^d} u(x,v')\,\D v' \mu_h^2\Big)
\end{align}
and corresponds to a simple relaxation towards the Maxwellian.
Denoting $Q_{\mathcal H}^*(h,\cdot)$ the formal adjoint of $Q_{\mathcal H}(h,\cdot)$, one can easily compute
\begin{align}\label{ql1}
Q_{\mathcal H}(h,\mathcal M_h^2)=0 \qquad \text{and } \qquad Q_{\mathcal H}^*(h,1)=0
\end{align}
so in particular $\mathcal M_h^2$ is a stable state of \eqref{Boltl1} and $Q_{\mathcal H}$ features the local conservation of mass.
In order to do a perturbative study of the time independent operator associated to \eqref{Boltl1} near $\mathcal M^2_h$, we introduce the natural Hilbert space 
$$\mathcal H= \big\{u \in \mathcal D' \,;\, \mathcal M_h^{-1} u\in L^2(\R^{2d})\big\}.$$
It is clear from the Cauchy Schwarz inequality that $\mathcal H$ is indeed a subset of $L^1(\R^{2d})$ provided that $\e^{-\frac{V}{2h}}\in L^2(\R^d_x)$.
In view of \eqref{ql1} and the definition of $\mathcal H$, it is more convenient to work with the new unknown
$$f=\mathcal M_h^{-1} u\, :\R_+\to L^2(\R^{2d})$$
for which the new equation becomes 
\begin{equation}\label{Boltgene}
\left\{
\begin{aligned}
&h\partial_tf+v\cdot h\partial_xf-\partial_xV\cdot h\partial_vf+Q_h(f)=0 \\ 
&f_{|t=0}=f_0
\end{aligned}
\right.
\end{equation}
where
\begin{align}\label{QHconj}
Q_h=\mathcal M_h^{-1}\circ Q_{\mathcal H}(h, \cdot) \circ \mathcal M_h.
\end{align}
Denoting with the notation \eqref{muh},
\begin{align}\label{Pih}
\Pi_h : L^2(\R^{2d})\to L^2(\R^{2d})
\end{align}
the orthogonal projection on $\mu_h\, L^2(\R^{d}_x)$,
we have by \eqref{QH} and \eqref{QHconj}
\begin{align}\label{qlr}
Q_h=h(\mathrm{Id}-\Pi_h).
\end{align}
Our study will be focused on the spectral properties of the new time independent operator 
\begin{align}\label{Ph}
P_h&=v\cdot h\partial_x-\partial_xV\cdot h\partial_v+h(\mathrm{Id}-\Pi_h) \nonumber \\
&=X_0^h+Q_h 
\end{align}
where the notation $X_0^h$ will stand for the operator $v\cdot h\partial_x-\partial_xV\cdot h\partial_v$, but also for the vector field $(x,v)\mapsto h(v,-\partial_x V(x))$.

This type of questions has recently known some major progress on the impulse of microlocal methods.
The operator $P_h$ was already studied in 2016 in \cite{Robbe} where the use of hypocoercive techniques enabled to get some resolvent estimates and establish a rough localization of its small spectrum which consists of exponentially small eigenvalues in correspondance with the minima of the potential $V$. 
This type of result is similar to the one obtained for example for the Witten Laplacian by Helffer and Sj\"ostrand in \cite{HelSjo} in the 1980's.
Such a localization already leads to return to equilibrium and metastability results which can be improved as the description of the small spectrum becomes more precise.
For example, sharp asymptotics of the small eigenvalues of the Witten Laplacian were obtained later in the 2000's in \cite{BoGaKl} and \cite{HeKlNi} and later again for Kramers-Fokker-Planck type operators by H\'erau et al. in \cite{HHS}.
In these papers, the idea was to exhibit a supersymmetric structure for the operator and then study both the derivative acting from 0-forms into 1-forms and its adjoint with the help of basic quasimodes.
However, these methods do not apply to the Boltzmann equation as in that case the matrix appearing in the modification of the inner product does not obey good estimates with respect to the semiclassical parameter $h$ (see for instance \cite{theserobbe} for the case of the \emph{mild relaxation} collision operator).

This is why our goal in this paper will be to give precise spectral asymptotics for the operator $P_h$ through a more recent approach which consists in directly constructing a family of accurate \emph{gaussian quasimodes} for our operator in the spirit of \cite{LPMichel,BonyLPMichel} for Fokker-Planck type differential operators and \cite{me} for the mild relaxation Boltzmann equation.
Here the first difficulty is that like in \cite{me}, the operator that we consider is non local and hence it is harder to compute its action on the constructed quasimodes.
This will be overcome thanks to the factorization result stated in Proposition \ref{facto}.
The second and main difficulty is that unlike in \cite{me}, the bad microlocal properties of $Q_h$ are such that its action on a gaussian quasimode as used in \cite{LPMichel,BonyLPMichel,me} does not yield a precise exponential, but rather a superposition of exponentials (see Lemma \ref{actionq}) wich will lead to the introduction of some new quasimodes given by a superposition of \og usual \fg{} gaussian quasimodes.
The result that we manage to establish is similar to the one from \cite{HeKlNi} for the Witten Laplacian as well as the ones from \cite{HHS, HHS11} with recent improvements by Bony et al. in \cite{BonyLPMichel} for the Fokker-Planck equation.

\subsection{Setting and main results}
For $d'\in \N^*$ and $Z\in \C^{d'} $, we use the standard notation $\langle Z \rangle=(1+|Z|^2)^{1/2}$.
Let us introduce a few notations of semiclassical microlocal analysis which will be used in all this paper.
These are mainly extracted from \cite{Zworski}, chapter 4.
For our purpose, it is sufficient to consider pseudo-differential operators acting only in the variable $v$.
We will denote $\eta\in \R^{d}$ 
the dual variable of $v$ 
and use the semiclassical Fourier transform
$$\mathcal F_h(f)(\eta)= \int_{\R^{d}}\e^{-\frac ih v\cdot \eta}f(v) \, \D v.$$
We consider the space of semiclassical symbols 
\begin{align*}
S^\kappa \big(\langle (v,\eta) \rangle^k\big)=\big\{a_h \in \mathcal C^\infty(\R^{2d}) \, ; \, \forall \alpha \in \N^{2d}, \exists \, C_\alpha>0\;
		\emph{\text{such that }} |\partial^\alpha a_h(v,\eta)|\leq C_\alpha h^{-\kappa |\alpha|}\langle (v,\eta) \rangle^k\big\}
\end{align*}
where $k\in \R$ and $\kappa \in [0, 1/2]$.
Given a symbol $a_h \in S^\kappa(\langle (v,\eta) \rangle^k)$, we define the associated semiclassical pseudo-differential operator for the Weyl quantization acting on functions $u \in \mathcal S(\R^{d})$ by
$$\mathrm{Op}_h(a_h)u(v)=(2\pi h)^{-d}\int_{\R^{d}} \int_{\R^{d}}\e^{\frac ih (v-v')\cdot \eta}a_h\Big(\frac{v+v'}{2},\eta\Big)u(v')\,\D v'\D \eta$$
where the integrals may have to be interpreted as oscillating integrals.
We will denote $\Psi^\kappa(\langle (v,\eta) \rangle^k)$ the set of such operators.
Note that 
the operator $\mathrm{Op}_h(a_h)$ admits the distributional kernel
$$K_h(v,v')=\mathcal F_h^{-1}\bigg(a_h\bigg(\frac{v+v'}{2},\cdot\bigg)\bigg)(v-v').$$
Conversely, if an operator $\mathrm{Op}_h(a_h)\in \Psi^\kappa(\langle (v,\eta) \rangle^k)$ admits the distributional kernel $K_h(v,v')$, then its symbol is given by
\begin{align}\label{FK}
a_h(v,\eta)=\mathcal F_h\Big(\big(K_h\circ A\big)(v,\cdot)\Big)(\eta)
\end{align}
where $A$ denotes the change of variables 
\begin{align}\label{mata}
A(v,v')=(v+v'/2,v-v'/2).
\end{align}

We will also make a few confining assumptions on the function $V$, assuring for instance that the bottom spectrum  of the associated Witten Laplacian is discrete. In particular, our potential will satisfy Assumption 2 from \cite{LPMichel} and Hypothesis 1.1 from \cite{Robbe}.
\begin{hypo}\label{V}
The potential $V$ is a smooth Morse function depending only on the space variable $x\in \R^d$ with values in $\R$ which is bounded from below and such that
$$|\partial_x V(x)|\geq \frac1C\qquad \text{for }|x|>C.$$
Moreover, for all $\alpha \in \N^d$ with $|\alpha|\geq 2$, there exists $C_\alpha$ such that 
$$|\partial_x^\alpha V|\leq C_\alpha.$$
\hop
In particular, for every $0\leq k\leq d$, the set of critical points of index $k$ of $V$ that we denote $\mathtt U^{(k)}$ is finite and we set 
\begin{align}\label{n0}
n_0=\# \mathtt U^{(0)}.
\end{align}
Finally, we will suppose that $n_0\geq 2$.
\end{hypo}
\hip
The last assumption comes from the fact that when $n_0=1$, the so-called \emph{small spectrum} of the operator $P_h$ (i.e its eigenvalues with exponentially small modulus) is trivial, so there is nothing to study.
It is shown in \cite{MeSc}, Lemma 3.14 that for a function $V$ satisfying Hypothesis \ref{V}, we have $V(x)\geq |x|/C$ outside of a compact. In particular, under Hypothesis \ref{V}, it holds $\e^{-V/2h}\in L^2(\R^{d}_x)$.
Moreover, in our setting, $X_0^h$ is a smooth vector field whose differential is bounded on $\R^{2d}$, so the operator $X_0^h$ endowed with the domain
$$D=\{u \in L^2(\R^{2d}) \, ; \, X_0^hu \in L^2(\R^{2d})\}$$
is skew-adjoint on $L^2(\R^{2d})$ and the set $\mathcal S(\R^{2d})$ is a core for this operator.
Since moreover the collision operator $Q_h$ defined in \eqref{qlr} is bounded and self-adjoint, we have $(P_h,D)^*=(-X_0^h+Q_h,D)$ and $(P_h,D)$ is m-accretive on $L^2(\R^{2d})$.

For an operator such as $P_h$, which is not for instance self-adjoint with compact resolvent, we do not have any information a priori on its spectrum (except here that it is contained in $\{z \in \C \, ; \, \mathrm{Re} \, z \geq 0\}$).
In \cite{Robbe}, the use of hypocoercive techniques enabled to establish a first description of the spectrum of $P_h$ near 0 which, in the
spirit of the case of other non self-adjoint operators studied in \cite{HHS}, appears in particular to be discrete.
More precisely, the following result is shown in \cite{Robbe}:
\begin{thm}\label{thmRobbe}
Assume that Hypothesis \ref{V} is satisfied and recall the notation \eqref{n0}.
Then the operator $(P_h,D)$ admits $0$ as a simple eigenvalue.
Moreover, there exists $c>0$ and $h_0>0$ such that for all $0<h\leq h_0$, we have that 
$\mathrm{Spec}(P_h)\cap \{\mathrm{Re} \,z\leq ch\}$ consists of exactly $n_0$ eigenvalues (counted with algebraic multiplicity) 
which are real and exponentially small with respect to $1/h$.
Finally, for all $0<\tilde c\leq c$, the resolvent estimate 
$$(P_h-z)^{-1}=O(h^{-1})$$
holds uniformly in $\{\mathrm{Re} \,z\leq ch\} \backslash B(0, \tilde c h)$.
\end{thm}
\hip

In order to study the long time behavior of the solutions of \eqref{Boltgene}, we need a precise description of the small spectrum of $P_h$.
To this aim, we construct in Section \ref{sectionquasim} a family of accurate quasimodes localized around the minima of $V$ that enables us to establish sharp asymptotics of the small eigenvalues of $P_h$.
This will lead to the following Theorem which is the main result of this paper.
Before we can state it, let us introduce a few notations that we will use throughout the paper.
We denote 
 \begin{align}\label{W}
W(x,v)=\frac{V(x)}{2}+\frac{v^2}{4}
\end{align}
the global potential on $\R^{2d}$ and for $x\in \R^d$,
\begin{align}\label{hessiennes}
\mathcal V_x \text{ (resp. }\mathcal W_x\text{) the Hessian of } V \text{ at }x\;\big( \text{resp. the Hessian of } W \text{ at }(x,0)\big).
\end{align}
When $\mathbf s\in \R^d$ is a saddle point of $V$ (i.e $\mathbf s\in \mathtt U^{(1)}$), we also denote
\begin{align}\label{tau}
\tau_{\mathbf s} \text{ the only negative eigenvalue of } \mathcal V_{\mathbf s}.
\end{align}
For the sake of simplicity, we will make in the statement of the Theorem an additionnal assumption (Hypothesis \ref{jvide}) on the topology of the potential $V$ that could actually be omitted (see \cite{Michel} or \cite{BonyLPMichel}).
It implies in particular that $V$ has a unique global minimum that we denote $\underline{\mathbf m}$.\\
According to Theorem \ref{thmRobbe}, we can associate to each $\mathbf m \in \mathtt U^{(0)}\backslash \{\underline{\mathbf m}\}$ a non zero exponentially small eigenvalue of $P_h$ that we denote $\lambda(\mathbf m,h)$.
\begin{thm}\label{thmToto}
Suppose that 
Hypotheses \ref{V} and \ref{jvide} are satisfied and recall the notations \eqref{hessiennes}-\eqref{tau}.\\
The exponentially small eigenvalues of $P_h$ satisfy the following equivalent in the limit $h\to 0$:
\begin{align}\label{rllambda}
\lambda(\mathbf m,h)\sim h \varrho(\mathbf m)\, \e^{\frac{-2S(\mathbf m)}{h}}
\end{align}
with
\begin{align}
\varrho(\mathbf m)=\frac{1}{\pi}\sum_{\mathbf s\in \mathbf j(\mathbf m)}\bigg(\frac{2+\sqrt2}{2-\sqrt2}\bigg)^{\frac{1}{\sqrt{|\tau_{\mathbf s}|}}}\bigg(\frac{\det \mathcal V_{\mathbf m}}{|\det \mathcal V_{\mathbf s}|}\bigg)^{1/2}   \int_{\gamma_1\leq z\leq \gamma < 1} k_0^{\mathbf s}(\gamma) k_0^{\mathbf s}(z) \ln \bigg(2\,\frac{(1+z)(1+\gamma)}{1+3z+3\gamma+z\gamma}\bigg) \, \D z \D \gamma
\end{align}
where
$$k^{\mathbf s}_0(z)=\frac{2\sqrt2}{\sqrt{|\tau_{\mathbf s}|}(z-\gamma_2)^2}\Big(\frac{z-\gamma_1}{z-\gamma_2}\Big)^{\frac{1}{2\sqrt{|\tau_{\mathbf s}|}}-1} \qquad ; \qquad \gamma_1=-3+2\sqrt2 \qquad ; \qquad \gamma_2=-3-2\sqrt2 $$
and the maps $S$ and $\mathbf j$ are defined in Definition \ref{j et s}.
\end{thm}

Finally, following \cite{Robbe}, we use the sharp localization obtained in Theorem \ref{thmToto} in order to discuss the phenomena of return to equilibrium and metastability for the solutions of \eqref{Boltgene}.
More precisely, we are able to give a sharp rate of convergence of the semigroup $\e^{-tP_h/h}$ towards $\Pro_1$, the orthogonal projector on Ker $P_h$ : denoting $\lambda^*$ the smallest non zero eigenvalue of $P_h$, we establish that the rate of return to equilibrium is essentially given by $\lambda^*/h$:
\begin{cor}\label{ral}
Under the assumptions of Theorem \ref{thmToto}, there exists $h_0>0$ such that for all $0<h\leq h_0$ and $t\geq 0$,
$$\|\e^{-tP_h/h}-\Pro_1\|\leq C \e^{-t \lambda^*/h}.$$
\end{cor}
\hip
Besides, in the spirit of \cite{BonyLPMichel,me}, we also show the metastable behavior of the solutions of \eqref{Boltgene}:
\begin{cor}\label{meta}
Suppose that the assumptions of Theorem \ref{thmToto} hold true.
Let us consider some local minima $\mathbf m_1=\underline{\mathbf m}$, $\mathbf m_2$, $\dots$, $\mathbf m_p$ such that 
$$S\big(\mathtt U^{(0)}\big)=\{+\infty=S(\mathbf m_1) > S(\mathbf m_2) > \dots > S(\mathbf m_p)\}$$
for the map $S$ from Definition \ref{j et s}.
For $2\leq k \leq p$, denote $\Pro_k$ the spectral projection (which is not necessarily orthogonal) associated to the eigenvalues of $P_h$ that are $O\big(\e^{-2\frac{S(\mathbf m_k)}{h}}\big)$.
Then for any times $(t_k^\pm)_{1\leq k \leq p}$ satisfying 
$$t_p^-\geq |\ln (h^\infty)| \quad \text{and }\quad t_k^-\geq |\ln (h^\infty)|\e^{2\frac{S(\mathbf m_{k+1})}{h}}\quad \text{for }\quad k=1,\dots, p-1$$
as well as
$$t_1^+=+\infty \qquad \text{and }\quad t_k^+=O\Big(h^\infty\e^{2\frac{S(\mathbf m_{k})}{h}}\Big)\qquad \text{for }\quad k=2,\dots, p$$
one has 
$$\e^{-tP_h/h}=\Pro_k+O(h^\infty)\qquad \text{on }[t_k^-,t_k^+].$$ 
\end{cor}
\hip
In other words, we have shown the existence of timescales on which, during its convergence towards the global equilibrium, the solution of \eqref{Boltgene} will essentially visit the metastable spaces associated to the small eigenvalues of $P_h$.

Another perspective would then be to study the case of collision operators satisfying the local conservation laws of physics, such as the \emph{full linear Boltzmann operator}
$$Q^{FL}_h=h(\mathrm{Id}-\Pi_h^{FL})$$
with $\Pi_h^{FL}$ the orthogonal projector on the \emph{collision invariants} subspace
$$\mathrm{Vect}_{\R^d_v}\, \Big\{\e^{-\frac{v^2}{4h}},\, v_1\e^{-\frac{v^2}{4h}},\, \dots , \, v_d\e^{-\frac{v^2}{4h}},\, v^2\e^{-\frac{v^2}{4h}}\Big\}\,L^2(\R^d_x)$$
which was recently studied in \cite{CaDoHeMiMoSc} at fixed temperature.

\section{Preliminaries} \label{prelim}
\hip
From now on, the letter $r$ will denote a small universal positive constant whose value may decrease as we progress in this paper (one can think of $r$ as $1/C$).

\subsection{Naive approach}

In order to investigate a first natural approach to our problem consisting in trying to reproduce the method from \cite{me} which was itself inspired by \cite{LPMichel,BonyLPMichel}, let us make for simplicity and for this subsection only an additional assumption.
\begin{hypo}\label{naive}
The potential $V$ has exactly one saddle point $\mathbf s$. 
\end{hypo}
\hip
Roughly speaking, this approach consists in introducing a linear form $\ell(x,v)=\ell_x\cdot(x-\mathbf s)+\ell_v\cdot v$ in the variables $(x-\mathbf s,v)$ as well as a gaussian cut-off $\theta$ which is essentially given by
$$\theta(x,v)=\int_0^{\ell(x,v)}\e^{-\frac{s^2}{2h}}\D s.$$
With the notation \eqref{W}, the idea is then to introduce the so-called gaussian quasimode
$$\varphi(x,v)=\theta(x,v)\;\e^{-\frac{W(x,v)}{h}}$$
and compute $P_h\varphi$ in order to then choose the linear form $\ell$ minimizing the norm of $P_h\varphi$.
We already know from \cite{me} (proof of Proposition 3.13) that 
\begin{align}\label{x0phi}
X_0^h\varphi(x,v)=h\,p_\ell(x,v)\e^{-\frac 1h \big(W(x,v)+\frac12\ell^2(x,v)\big)}\big(1+O(h)\big)\qquad \text{with }p_\ell=O_{L^\infty}(1),\quad |x-\mathbf s|,|v|<r.
\end{align}
It is also shown that the collision operator studied in this reference, that we denote $Q_h^{S^0}$, satisfies a similar result:
\begin{align}\label{qs0}
Q^{S^0}_h\varphi(x,v)=h\,q_\ell(x,v)\e^{-\frac 1h \big(W(x,v)+\frac12\ell^2(x,v)\big)}\big(1+O(h)\big)\qquad \text{with }q_\ell=O_{L^\infty}(1),\quad |x-\mathbf s|,|v|<r
\end{align}
and it is then sufficient in that case to choose $\ell$ so that $p_\ell=-q_\ell$.\\
In our case, although $Q_h$ may appear as a quite simple operator as it is just an orthogonal projection, in order to perform a computation similar to \eqref{qs0}, it will be more convenient to adopt a microlocal point of view.
This is the point of the two following results which are proven in Appendix \ref{qmicro}.

\begin{prop}\label{facto}
Let us denote 
\begin{align}\label{bh}
b_h=h\partial_v+v/2.
\end{align}
There exists a symbol $m_h\in S^{1/2}\big(\langle v, \eta \rangle^{-2}\big)$ given by
$$m_h(v,\eta)=2\int_0^1 (y+1)^{d-2} \e^{-\frac{y}{h}\big(\frac{v^2}{2}+2\eta^2\big)} \D y$$
such that 
$$Q_h=b_h^*\circ \mathrm{Op}_h(m_h \, \mathrm{Id})\circ b_h.$$
\end{prop}

\begin{cor}\label{opgb}
One has
$$Q_h= \mathrm{Op}_h(g_h )\circ b_h$$
with 
$$g_h(v,\eta)=\int_0^1 (y+1)^{d-1} \e^{-\frac{y}{h}\big(\frac{v^2}{2}+2\eta^2\big)} \D y\,(-2i\eta^t+v^t)\in S^{1/2}\big(\langle v, \eta \rangle^{-1}\big).$$
\end{cor}
\hip
We are now in position to establish the following fundamental computation which shows that the balancing obtained between $X_0^h\varphi$ and $Q_h^{S^0}\varphi$ cannot happen between $X_0^h\varphi$ and $Q_h\varphi$.
This will motivate the introduction of some new quasimodes later on.
\begin{lem}\label{actionq}
Assume for simplicity that Hypothesis \ref{naive} holds true and let $\ell$ a linear form in the variables $(x-\mathbf s,v)$.
We have
$$Q_h\varphi(x,v)
=-h\int_0^1 \partial_y(L_y)\,\e^{-\frac{W(x,v)+\frac12 L_y^2(x,v)}{h}}\,\D y\,\cdot \begin{pmatrix}x-\mathbf s\\v\end{pmatrix}$$
where with a slight abuse of notations, $L_y$ denotes both the linear form
$$L_y(x,v)=\frac{(1+y) \ell_x\cdot(x-\mathbf s)+(1-y)\ell_v\cdot v}{\Big(4y\ell_v^2+(y+1)^2\Big)^{1/2}}$$
and the vector representing it.
Moreover, denoting 
\begin{align}\label{my}
m_{y,h}(v,\eta)=2 (y+1)^{d-2} \e^{-\frac{y}{h}\big(\frac{v^2}{2}+2\eta^2\big)},
\end{align}
we have
\begin{align}\label{opm}
\mathrm{Op}_h(m_{y,h}) \circ b_h \varphi(x,v)&=2h(2\pi h)^{-d/2}\e^{-\frac{V(x)}{2h}} \frac{(y+1)^{d-2}}{(4y)^{\frac d2}}\\
		&\qquad \qquad \qquad \qquad\times \int_{v'\in \R^d} \e^{-\frac 1h \big(\frac{v'^2}{4}+\frac y8 (v+v')^2+\frac{(v-v')^2}{8y}+\frac12\ell^2(x,v')\big)} \,\D v' \, \ell_v.
\end{align}
\end{lem}

\begin{proof}
According to Corollary \ref{opgb}, we have
\begin{align}\label{opgphaseell}
Q_h\varphi(x,v)&=\mathrm{Op}_h(g_h)\Big[h\partial_v\theta \e^{-W/h}\Big](x,v)\nonumber\\
	&=h\mathrm{Op}_h(g_h)\Big[\e^{-\frac 1h \big(W+\frac12\ell^2\big)}\ell_v\Big](x,v)\\
	&=h(2\pi h)^{-d}\int_{v'\in \R^d} \int_{\eta\in\R^d} \e^{\frac ih (v-v')\cdot\eta} g_h\Big(\frac{v+v'}{2},\eta\Big) \e^{-\frac 1h \big(W(x,v')+\frac12\ell^2(x,v')\big)} \,\D v' \D\eta\, \ell_v.\nonumber
\end{align}
Let us now compute the integral in $\eta$ with the expression of $g_h$ from Corollary \ref{opgb}:
\begin{align}
\int_{\eta\in \R^d} \e^{\frac{i}{h}(v-v')\cdot\eta} g_h\Big(\frac{v+v'}{2},\eta\Big) \D \eta=\int_0^1 (y+1)^{d-1} \e^{-\frac{y(v+v')^2}{8h}} \bigg[\frac{(v+v')^t}{2} \int_{\eta\in \R^d} \e^{\frac{i}{h}(v-v')\cdot\eta} \e^{-\frac{2y\eta^2}{h}} \D \eta\nonumber\\
		-2i\int_{\eta\in \R^d} \eta^t \e^{\frac{i}{h}(v-v')\cdot\eta} \e^{-\frac{2y\eta^2}{h}}\D\eta\bigg] \D y\\
	=\int_0^1 (y+1)^{d-1} \e^{-\frac{y(v+v')^2}{8h}} \bigg[\frac{(v+v')^t}{2}+\frac{(v-v')^t}{2y}\bigg] \int_{\eta\in \R^d} \e^{\frac{i}{h}(v-v')\cdot\eta} \e^{-\frac{2y\eta^2}{h}} \D \eta \D y\\
	=2(2\pi h)^{d/2}\int_0^1 \frac{(y+1)^{d-1}}{(4y)^{\frac d2 +1}}  \Big((v+v')y+v-v'\Big)^t \e^{-\frac{1}{8h}\big(y(v+v')^2+\frac{(v-v')^2}{y}\big)} \D y.
\end{align}
Hence, 
we get
\begin{align}\label{intv'}
Q_h\varphi(x,v)&=2h(2\pi h)^{-d/2}\e^{-\frac{V(x)}{2h}}\int_0^1 \frac{(y+1)^{d-1}}{(4y)^{\frac d2 +1}}\int_{v'\in \R^d}\Big((v+v')y+v-v'\Big)\\
		&\qquad \qquad \qquad \qquad\qquad \qquad\times  \e^{-\frac 1h \big(\frac{v'^2}{4}+\frac y8 (v+v')^2+\frac{(v-v')^2}{8y}+\frac12\ell^2(x,v')\big)} \,\D v' \D y\cdot \ell_v
\end{align}
and \eqref{opm} is now a straightforward adaptation of \eqref{intv'} with $m_{y,h}$ instead of $g_h$.
Denoting $x_{\mathbf s}=x-\mathbf s$,
$$M_y=\frac12\mathrm{Id} +\ell_v\ell_v^t+\frac{y^2+1}{4y}\mathrm{Id}\qquad \text{and}\qquad u_y(x_{\mathbf s},v)=\ell_x\cdot x_{\mathbf s}\,\ell_v+\frac{y^2-1}{4y}v,$$
\eqref{intv'} becomes by the change of variables $w=v'+M_y^{-1}u_y(x_{\mathbf s},v)$
\begin{align}\label{qphi}
Q_h\varphi(x,v)&=2h(2\pi h)^{-d/2}\e^{-\frac{V(x)}{2h}} \int_0^1 \frac{(y+1)^{d-1}}{(4y)^{\frac d2 +1}}\nonumber\\  
	&\qquad \times\exp\bigg[\frac{-1}{2h}\Big(\ell_x\ell_x^t x_{\mathbf s}\cdot x_{\mathbf s}+\frac{y^2+1}{4y} v^2-M_y^{-1} u_y(x_{\mathbf s},v)\cdot u_y(x_{\mathbf s},v)\Big)\bigg]\nonumber\\
	&\qquad  \times \int_{w\in \R^d} \Big[\Big(v-M_y^{-1}u_y(x_{\mathbf s},v)\Big)y+v+M_y^{-1}u_y(x_{\mathbf s},v)
	\Big]\e^{-\frac{M_y w\cdot w}{2h}}\D w  \D y\cdot \ell_v\nonumber\\
	&=2h\e^{-\frac{V(x)}{2h}} \int_0^1 \frac{(y+1)^{d-1}}{(4y)^{\frac d2 +1}}\det(M_y)^{-1/2}\Big((1+y)v+(1-y)M_y^{-1}u_y(x_{\mathbf s},v)\Big)\cdot \ell_v\\
	&\qquad\times\exp\bigg[\frac{-1}{2h}\Big(\ell_x\ell_x^t x_{\mathbf s}\cdot x_{\mathbf s}+\frac{y^2+1}{4y} v^2-M_y^{-1} u_y(x_{\mathbf s},v)\cdot u_y(x_{\mathbf s},v)\Big)\bigg]\,\D y
\end{align}
Now 
$$\frac{(y+1)^{d-1}}{(4y)^{\frac d2 +1}}\det(M_y)^{-1/2}=\frac{1}{4y\Big(4y\ell_v^2+(y+1)^2\Big)^{1/2}}$$
while 
\begin{align}\label{Mlv}
 M_y^{-1}\ell_v=\frac{4y}{4y\ell_v^2+(y+1)^2}\,\ell_v
\end{align}
so the prefactor in the integral from \eqref{qphi} becomes
$$\frac{1}{4y\Big(4y\ell_v^2+(y+1)^2\Big)^{1/2}}\bigg[\frac{4y(1-y)\ell_v^2}{4y\ell_v^2+(y+1)^2}\ell_x\cdot x_{\mathbf s}+\bigg((1+y)+\frac{(1-y)(y^2-1)}{4y\ell_v^2+(y+1)^2}\bigg)\ell_v\cdot v\bigg]$$
which is further equal to
\begin{align}\label{dyl}
\frac{(1-y)\ell_v^2 \ell_x\cdot x_{\mathbf s}+(1+y)(1+\ell_v^2)\ell_v\cdot v}{\Big(4y\ell_v^2+(y+1)^2\Big)^{3/2}}=-\frac12 \partial_y(L_y) \cdot \begin{pmatrix}x_{\mathbf s}\\v\end{pmatrix}.
\end{align}
Thus, it only remains to show that the exponentials coincide, i.e
$$\ell_x\ell_x^t x_{\mathbf s}\cdot x_{\mathbf s}+\frac{y^2+1}{4y} v^2-M_y^{-1} u_y(x_{\mathbf s},v)\cdot u_y(x_{\mathbf s},v)=\frac{v^2}{2}+L_y^2(x,v)$$
or equivalently
\begin{align}\label{egalphase}
\ell_x\ell_x^t x_{\mathbf s}\cdot x_{\mathbf s}+\frac{(y-1)^2}{4y} v^2-M_y^{-1} u_y(x_{\mathbf s},v)\cdot u_y(x_{\mathbf s},v)=
\frac{\Big((1+y)\ell_x\cdot x_{\mathbf s}+(1-y)\ell_v\cdot v\Big)^2}{4y\ell_v^2+(y+1)^2}.
\end{align}
Using \eqref{Mlv}, we already obtain 
$$M_y^{-1} u_y(x_{\mathbf s},v)\cdot u_y(x_{\mathbf s},v)=\frac{4y\ell_v^2}{4y\ell_v^2+(y+1)^2}\ell_x\ell_x^t x_{\mathbf s}\cdot x_{\mathbf s}+2\frac{y^2-1}{4y\ell_v^2+(y+1)^2}\ell_x\cdot x_{\mathbf s}\ell_v\cdot v+\frac{(y^2-1)^2}{16y^2}M_y^{-1}v\cdot v$$
so the LHS of \eqref{egalphase} becomes
\begin{align}\label{lhs}
\frac{(1+y)^2}{4y\ell_v^2+(y+1)^2}(\ell_x\cdot x_{\mathbf s})^2+2\frac{1-y^2}{4y\ell_v^2+(y+1)^2}\ell_x\cdot x_{\mathbf s}\ell_v\cdot v+\Big(\frac{(y-1)^2}{4y}-\frac{(y^2-1)^2}{16y^2}M_y^{-1}\Big)v\cdot v.
\end{align}
Finally, still using \eqref{Mlv}, one can easily check that
$$\frac{(y-1)^2}{4y}-\frac{(y^2-1)^2}{16y^2}M_y^{-1}=\frac{(1-y)^2}{4y\ell_v^2+(y+1)^2}\ell_v\ell_v^t$$
so \eqref{lhs} equals the RHS of \eqref{egalphase} and the proof is complete.
\end{proof}

This result shows that unlike in the case of some $S^0$ collisions operators as studied in \cite{me} (or even in the case of differential operators \cite{LPMichel,BonyLPMichel}), here the action of $Q_h$ on the quasimode $\varphi$ does not yield a precise exponential, but rather a superposition of exponentials with the linear form in the phase varying.
This suggests the introduction of some new quasimodes given by a superposition of functions similar to $\varphi$ with the linear form varying.

\subsection{Labeling of the potential minima}
We now drop Hypothesis \ref{naive}.
Before we can construct our quasimodes, we need to recall the general labeling of the minima which originates from \cite{BoGaKl,HeKlNi} and was generalized in \cite{HHS11}, as well as the topological constructions that go with it.
Here we only introduce the essential objects and omit the proofs.
For more details, we refer to \cite{me} where it is in particular shown that, roughly speaking, the constructions for the potential $V/2$ are the projections on $\R^d_x$ of the ones for the global potential $W$.
Recall that we denote 
\begin{align}\label{UkY}
\mathtt U^{(k)} \text{ the critical points of }V \text{ of index }k.
\end{align}
For shortness, we will write \og CC \fg{} instead of \og connected component \fg{}.
The constructions rely on the following fundamental observation which is an easy consequence of the Morse Lemma (see for instance \cite{me}, Lemma 3.1 for a proof):
\begin{lem}\label{1.4}
If $x\in \mathtt U^{(1)}$, then there exists $r_0>0$ such that for all $0<r<r_0$, $(x,0)$ has a connected neighborhood $\mathcal O_r$ in $B_0(x,r)$ such that $\mathcal O_r\cap \{W< W(x,0)\}$ has exactly 2 CCs.\\
\end{lem}
\hip
It motivates the following definition:
\begin{defi}\label{ssv}
\begin{enumerate}
\item We say that $x\in \mathtt U^{(1)}$ is a separating saddle point and we denote $x\in \mathtt V^{(1)}$ if for every $r>0$ small enough, the two CCs of $\mathcal O_r\cap \{W< W(x,0)\}$ are contained in different CCs of $\{W< W(x,0)\}$.
\item We say that $\sigma \in \R$ is a separating saddle value if $\sigma \in \frac V2(\mathtt V^{(1)})$.
\end{enumerate}
\end{defi}

It is known (see for instance \cite{me}, Lemma 3.4) that $\mathtt V^{(1)}\neq \emptyset$ since $n_0\geq 2$. Let us then denote $\sigma_2>\dots>\sigma_N$ where $N\geq 2$ the different separating saddle values of $V/2$ and for convenience we set $\sigma_1=+\infty$.
For $\sigma \in \R\cup\{+\infty\}$, let us denote $\mathcal C_\sigma$ the set of all the CCs of $\{W<\sigma\}$.
We call \emph{labeling} of the minima of $V$ any injection $\mathtt U^{(0)}\to \llbracket 1,N \rrbracket \times \N^*$ which we denote for shortness $(\mathbf m_ {k,j})_{k,j}$.
Given a labeling $(\mathbf m_{k,j})_{k,j}$ of the minima of $V$, we denote for $k\in \llbracket 1,N \rrbracket$ 
$$\mathtt U^{(0)}_k=\big\{\mathbf m_{k',j}\,;\, 1\leq k'\leq k\big\}\cap \Big\{\frac V2<\sigma_k\Big\}$$
and we say that the labeling is \emph{adapted} to the separating saddle values if for all $k\in \llbracket 1,N \rrbracket$, each $\mathbf m_{k,j}$ is a global minimum of $V$ restricted to some CC of $\{V/2<\sigma_k\}$ and the map 
\begin{align}\label{TkY}
T_k: \mathtt U^{(0)}_k \to \mathcal C_ {\sigma_k}
\end{align}
sending $\mathbf m \in \mathtt U^{(0)}_k$ on the element of $\mathcal C_ {\sigma_k}$ containing $(\mathbf m,0)$ is bijective.
In particular, it implies that each $\mathbf m_{k,j}$ belongs to $\mathtt U^{(0)}_k$.
Such labelings exist, one can for instance easily check that the usual labeling procedure presented in \cite{HHS11} is adapted to the separating saddle values.
From now on, we fix a labeling $(\mathbf m_{k,j})_ {k,j}$ adapted to the separating saddle values of $V$.

\begin{defi}\label{j et s}
Recall the notation \eqref{UkY} and Definition \ref{ssv}.
We define the following mappings:
\begin{enumerate}[label=\textbullet]
\item $E: \mathtt U^{(0)}\xrightarrow{{}\quad{}} \mathcal P(\R^{2d})$\\
	${} \; \mathbf m_{k,j} \xmapsto{{}\quad{}} T_k(\mathbf m_{k,j})$\\
where $T_k$ is the map defined in \eqref{TkY}.
\item 
$\mathbf j^W:\mathtt U^{(0)}\to \mathcal P\Big(\big(\mathtt V^{(1)}\cup \{\mathbf s_1\}\big) \times \{0\}\Big)$\\
given by $\mathbf j^W(\mathbf m_{1,1})=(\mathbf s_1,0)$ where $\mathbf s_1$ is a fictive saddle point such that $V(\mathbf s_1)=\sigma_1=+\infty$; and for $2\leq k\leq N$, $\mathbf j^W(\mathbf m_{k,j})=\partial E(\mathbf m_{k,j})\cap \big(\mathtt V^{(1)}\times \{0\}\big)$ which is not empty (see for instance Lemma 3.5 from \cite{me}), finite and included in $\{W=\sigma_k\}$.
\item 
$\mathbf j:\mathtt U^{(0)}\to \mathcal P\big(\mathtt V^{(1)}\cup \{\mathbf s_1\}\big)$\\
such that $\mathbf j(\mathbf m)\times\{0\}=\mathbf j^W(\mathbf m)$.
\item
$\boldsymbol \sigma:\mathtt U^{(0)}\to \frac V2(\mathtt V^{(1)})\cup \{\sigma_1\}$\\
${} \quad \mathbf m \mapsto \frac V2(\mathbf j(\mathbf m))$\\
where we allow ourselves to identify the set $\frac V2(\mathbf j(\mathbf m))$ and its unique element in $\frac V2(\mathtt V^{(1)})\cup \{\sigma_1\}$.
\item
$S: \mathtt U^{(0)}\xrightarrow{{}\quad{}} ]0,+\infty]$\\
	${} \quad \mathbf m \xmapsto{{}\quad{}} \boldsymbol \sigma(\mathbf m)-\frac V2(\mathbf m)$.
\end{enumerate}
\end{defi}
\hip
Following \cite{BoGaKl, HeKlNi, HHS11, LPMichel}, we can now state our last assumption that allows us to treat the generic case.
As mentionned in the introduction, this assumption could actually be omitted (see \cite{Michel} or \cite[section 6]{BonyLPMichel}) but this would introduce additionnal difficulties that are not the main concern of this paper.

\begin{hypo}\label{jvide}
Recall that we fixed a labeling $(\mathbf m_{k,j})_ {k,j}$ adapted to the separating saddle values of $V$.
We assume the following:
\begin{enumerate}[label=\alph*)]
\item Each $\mathbf m_{k,j}$ is the only global minimum of $V$ on the CC of $\{\frac V2<\sigma_k\}$ to which it belongs. \label{mseul}
\item For all $\mathbf m\neq\mathbf m' \in \mathtt U^{(0)}$, the sets $\mathbf j(\mathbf m)$ and $\mathbf j(\mathbf m')$ do not intersect.\label{jvide2}
\end{enumerate}
\end{hypo}
\hip
According to Proposition 3.9 from \cite{me}, this hypothesis is equivalent to the facts that $(\mathbf m,0)$ is the only global minimum of $W|_{E(\mathbf m)}$ and $\mathbf j^W(\mathbf m)\cap\mathbf j^W(\mathbf m')=\emptyset$ which is what we use in practice.

\section{Accurate quasimodes} \label{sectionquasim}

\subsection{Gaussian quasimodes superposition}
By Hypothesis \ref{jvide}, the potential $V$ has a unique global minimum that we denote $\underline{\mathbf m}$.
For $r>0$, denote $\tilde r$ a positive number such that for all $\mathbf m  \in \mathtt U^{(0)}\backslash\{\underline{\mathbf m}\}$ and $\mathbf s \in \mathbf j(\mathbf m)$, 
\begin{align}\label{rtilde}
W(x,v)\geq \boldsymbol \sigma(\mathbf m)+\frac{r^2}{8}+\frac{v^2-r^2}{4} \quad \text{as soon as }|x-\mathbf s|<\tilde r \text{ and }|v|\geq r.
\end{align}
We also denote for $x\in \R^d$
\begin{align}\label{b0}
B_0(x,r)=B(x,\tilde r)\times B(0,r)\subseteq \R^{2d}.
\end{align}
Let $\mathbf m  \in \mathtt U^{(0)}\backslash\{\underline{\mathbf m}\}$; for each $\mathbf s \in \mathbf j(\mathbf m)$ we introduce a 
vector $\ell^{\mathbf s}=(\ell^{\mathbf s}_x,\ell^{\mathbf s}_v)\in \R^{2d}$ which will represent a linear form involved in the construction of our quasimodes. 
Note that thanks to item \ref{jvide2} from Hypothesis \ref{jvide}, each $\ell^{\mathbf s}$ corresponds to a unique $\mathbf m  \in \mathtt U^{(0)}\backslash\{\underline{\mathbf m}\}$.
In the spirit of \cite{LPMichel,BonyLPMichel,me} and more precisely in view of \eqref{x0phi}-\eqref{qs0}, we want $\mathbf s$ to be a local minimum of the function $W(x,v)+(\ell^{\mathbf s}_x\cdot (x-\mathbf s)+\ell^{\mathbf s}_v\cdot v)^2/2$ so according to Lemma \ref{det} and using the notation \eqref{hessiennes}, we take $\ell^{\mathbf s}$ satisfying
$$-\mathcal V_{\mathbf s}^{-1}\ell^{\mathbf s}_x\cdot \ell^{\mathbf s}_x-|\ell_v^{\mathbf s}|^2>\frac12.$$
This condition would be sufficient to develop a framework for the construction of our quasimodes.
However, it would appear later on when establishing a result analogous to the one of Lemma \ref{sisol} that the optimal choice of $\ell^{\mathbf s}$ would actually satisfy
$$-\mathcal V_{\mathbf s}^{-1}\ell^{\mathbf s}_x\cdot \ell^{\mathbf s}_x-|\ell_v^{\mathbf s}|^2=1.$$
Similarly, one could show in this framework from the analogous of \eqref{thetainte} that our quasimodes would not depend on the norm of $\ell^{\mathbf s}$.
Thus, we set  
\begin{align}\label{lv2}
|\ell^{\mathbf s}_v|^2=1
\end{align}
as well as
\begin{align}\label{-Vlxlx}
-\mathcal V_{\mathbf s}^{-1}\ell^{\mathbf s}_x\cdot\ell^{\mathbf s}_x=2
\end{align}
straight away as it leads to significant simplifications in the study.\\
We now introduce the polynomial
\begin{align}\label{P}
P(\gamma)=4\gamma+(\gamma+1)^2
\end{align}
and its two roots
$$\gamma_1=-3+2\sqrt2\in(-1,0)\qquad \text{and}\qquad \gamma_2=-3-2\sqrt2<-1.$$
In the spirit of Lemma \ref{actionq}, we also introduce for $\gamma\in (\gamma_1,1]$ the 
vector $(L^{\mathbf s}_{\gamma;x}\, ,\, L^{\mathbf s}_{\gamma;v})\in \R^{2d}$ 
where
\begin{align}\label{defL}
L^{\mathbf s}_{\gamma;x}=\frac{1+\gamma}{P(\gamma)^{1/2}}\,\ell_x^{\mathbf s}\qquad\text{and }\qquad L^{\mathbf s}_{\gamma;v}=\frac{1-\gamma}{P(\gamma)^{1/2}}\,\ell_v^{\mathbf s}.
\end{align}
Note that $(L^{\mathbf s}_{0;x}\, ,\, L^{\mathbf s}_{0;v})=\ell^{\mathbf s}$.
Lemma \ref{actionq} would actually suggest to consider only $\gamma\in [0,1]$, but doing so it would appear with the notation \eqref{K} that \eqref{ode} has no non-trivial solution, which is not true anymore when working on $(\gamma_1,1]$.
We do not consider $\gamma$ outside $(\gamma_1,1]$ as it would add a condition similar to \eqref{choixl} which would be incompatible with \eqref{choixl}.
Here is the picture of an example in the case $d=1$:
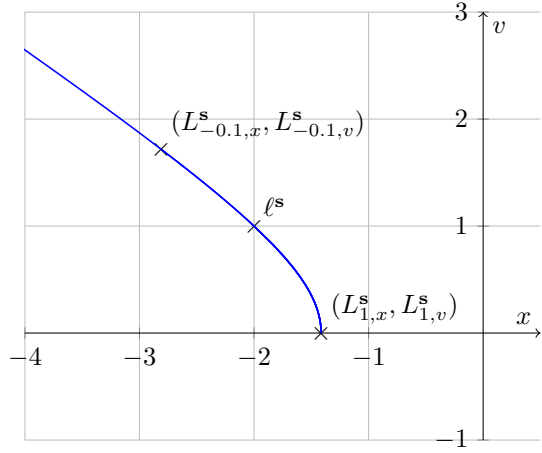
\begin{center}
\begin{tikzpicture}
    \begin{axis}[
            xmin=-4,xmax=0.5,
            ymin=-1,ymax=3,
            grid=both,
            ]
\addplot [color=blue,semithick,domain=1:0,samples=20]({-2*(1+x)/((x^2+6*x+1)^(1/2))},{(1-x)/((x^2+6*x+1)^(1/2))}) node[color=black] {$\times$};
\addplot [color=blue,semithick,domain=0:-0.1,samples=20]({-2*(1+x)/((x^2+6*x+1)^(1/2))},{(1-x)/((x^2+6*x+1)^(1/2))}) node[color=black] {$\times$};
\addplot [color=blue,semithick,domain=-0.1:-0.14,samples=20]({-2*(1+x)/((x^2+6*x+1)^(1/2))},{(1-x)/((x^2+6*x+1)^(1/2))});
\addplot [color=blue,semithick,domain=1:0,samples=20]({-2*(1+x)/((x^2+6*x+1)^(1/2))},{(1-x)/((x^2+6*x+1)^(1/2))}) node[color=black,above right] {$\ell^{\mathbf s}$};
\addplot [color=blue,semithick,domain=0:-0.1,samples=20]({-2*(1+x)/((x^2+6*x+1)^(1/2))},{(1-x)/((x^2+6*x+1)^(1/2))}) node[color=black,above right] {$(L^{\mathbf s}_{-0.1,x},L^{\mathbf s}_{-0.1,v})$};
\addplot [color=blue,semithick,domain=0:1,samples=20]({-2*(1+x)/((x^2+6*x+1)^(1/2))},{(1-x)/((x^2+6*x+1)^(1/2))}) node[color=black] {$\times$};
\addplot [color=blue,semithick,domain=0:1,samples=20]({-2*(1+x)/((x^2+6*x+1)^(1/2))},{(1-x)/((x^2+6*x+1)^(1/2))}) node[color=black,above right] {$(L^{\mathbf s}_{1,x},L^{\mathbf s}_{1,v})$};
    \end{axis}
\end{tikzpicture}
\end{center}
Since by Hypothesis \ref{V} we have that $V$ is a Morse function, there exists according to the Morse Lemma a smooth diffeomorphism $\phi_{\mathbf s}$ defined on $B(\mathbf s, \tilde r)$, sending $\mathbf s$ on $0$, whose differential at $\mathbf s$ is the identity and such that
\begin{align}\label{phi}
V\circ \phi^{-1}_{\mathbf s} =V(\mathbf s)+\frac12 \langle \mathcal V_{\mathbf s} \, \cdot, \cdot \rangle.
\end{align}
For shortness, we will use for $x\in B(\mathbf s, \tilde r)$ the notation
\begin{align}\label{xtilde}
\tilde x_{\mathbf s}=\phi_{\mathbf s}(x)
\end{align}
and we introduce the smooth function $L_\gamma^{\mathbf s}$ supported in $B(\mathbf s,2\tilde r)\times \R^d_v$ and given when $x$ is close to $\mathbf s$ 
by the twisted linear form:
$$L_\gamma^{\mathbf s}(x,v)=L^{\mathbf s}_{\gamma;x}\cdot \tilde x_{\mathbf s}+L^{\mathbf s}_{\gamma;v}\cdot v\;\quad \text{for } (x,v)\in B(\mathbf s,\tilde r)\times \R^d_v.$$
Now, let us denote $\zeta \in \mathcal C^{\infty}_c(\R, [0,1])$ an even cut-off function supported in $[-\delta,\delta]$ that is equal to $1$ on $[-\delta/2,\delta/2]$ where $\delta>0$ is a parameter to be fixed later.
As we will not be able to produce some remainder terms that are uniform with respect to $\gamma\in (\gamma_1,1]$, we will work on $[\gamma_1+\nu,1]$ with 
$$\nu>0 \text{ that will be fixed small enough before letting }h\to0.$$
Consider also a probability density $k^{\mathbf s}_\nu$ on $[\gamma_1+\nu,1]$ as well as the quantity
\begin{align}\label{approxa}
A^{\mathbf s}_{\nu,h}=\int_{\gamma_1+\nu}^1 k^{\mathbf s}_\nu(\gamma)\int_0^{\infty} \zeta\Big(\frac{s}{N^{\mathbf s}(\gamma)}\Big) \e^{-\frac{s^2}{2h}}\D s\, \D \gamma=
\frac{\sqrt{\pi h}}{\sqrt 2}\big(1+O(\e^{-\alpha/h})\big)\qquad \text{for some }\alpha>0,
\end{align}
where
$$N^{\mathbf s}(\gamma)=\Big(|L^{\mathbf s}_{\gamma;x}|^2+|L^{\mathbf s}_{\gamma;v}|^2\Big)^{1/2}\geq \frac1C.$$
We will also use the notation
$$U_\gamma^{\mathbf s}=\frac{L_\gamma^{\mathbf s}}{N^{\mathbf s}(\gamma)}.$$
We now define for each $\mathbf m  \in \mathtt U^{(0)}\backslash\{\underline{\mathbf m}\}$ the \emph{gaussian cut-off superposition} $\theta^{\mathbf m}_{\nu,h}$ as follows: if $(x,v)$ belongs to
\begin{align}
 \bigcup_{\gamma \in [\gamma_1+\nu,1]}\{|U_\gamma^{\mathbf s}|\leq 2\delta \} \cap B_0(\mathbf s,r)
\end{align}
for some $\mathbf s \in \mathbf j(\mathbf m)$, then 
\begin{align}\label{thetainte}
\theta^{\mathbf m}_{\nu,h}(x,v)=\frac12 \bigg(1+(A_{\nu,h}^{\mathbf s})^{-1}\int_{\gamma_1+\nu}^1 k^{\mathbf s}_\nu(\gamma)\int_0^{L_\gamma^{\mathbf s}(x,v)}\zeta\Big(\frac{s}{N^{\mathbf s}(\gamma)}\Big)\e^{-s^2/2h}\D s\, \D \gamma \bigg).
\end{align}
Here are some pictures of the set $\{|U_\gamma^{\mathbf s}|\leq 2\delta \} \cap B_0(\mathbf s,r)$ for $\gamma=\gamma_1+\nu$; $\gamma=0$ and $\gamma=1$:
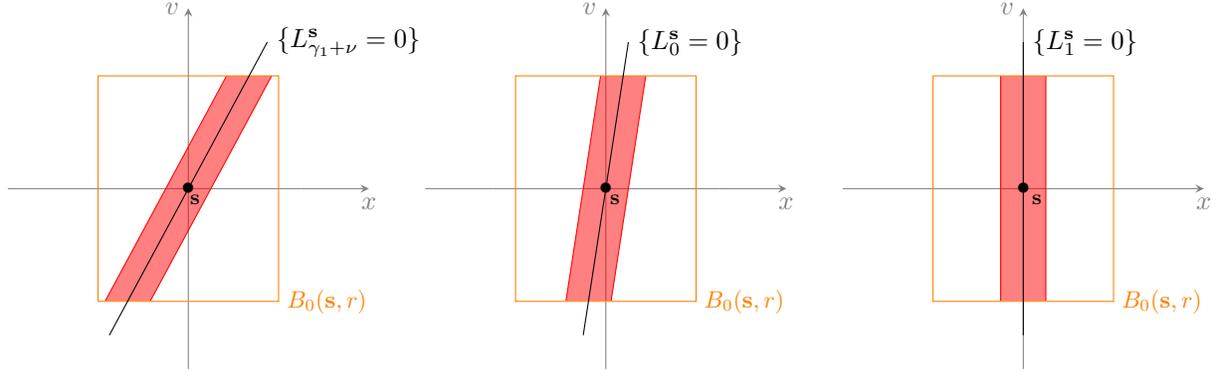
\begin{center}
\begin{tikzpicture}[scale=3]

\draw[gray] (-0.5,0) -- (0.5,0) ;
\draw[gray] (0,-0.5) -- (0,0.5) ;

\filldraw[red,fill opacity=0.5] (0,-0.65)--(0.2,0.65)--(0,0.65)--(-0.2,-0.65)-- cycle;

\draw (0,0) node {$\bullet$};

\fill[white] (-0.8,0.8) rectangle (0.8,0.5);
\fill[white] (-0.8,-0.8) rectangle (0.8,-0.5);

\draw (0.04,-0.05) node[scale=0.8] {$\mathbf s$};

\draw[gray] (-0.8,0) -- (-0.5,0) ;
\draw[gray] (0,-0.8) -- (0,-0.5) ;
\draw[gray,>=stealth,->] (0.5,0) -- (0.8,0) node[below] {$x$};
\draw[gray,>=stealth,->] (0,0.5) -- (0,0.8) node[left] {$v$};

\draw[orange] (-0.4,0.5) rectangle (0.4,-0.5) node[scale=0.9,right] {$B_0(\mathbf s,r)$};

\draw (-0.1,-0.65)--(0.1,0.65) node[right] {$\{L^{\mathbf s}_0=0\}$};

\begin{scope}[xshift=1.85 cm]  
\draw[gray] (-0.5,0) -- (0.5,0) ;
\draw[gray] (0,-0.5) -- (0,0.5) ;

\filldraw[red,fill opacity=0.5] (0.1,-0.65)--(0.1,0.65)--(-0.1,0.65)--(-0.1,-0.65)-- cycle;

\draw (0,0) node {$\bullet$};

\fill[white] (-0.8,0.8) rectangle (0.8,0.5);
\fill[white] (-0.8,-0.8) rectangle (0.8,-0.5);

\draw (0.06,-0.05) node[scale=0.8] {$\mathbf s$};

\draw[gray] (-0.8,0) -- (-0.5,0) ;
\draw[gray] (0,-0.8) -- (0,-0.5) ;
\draw[gray,>=stealth,->] (0.5,0) -- (0.8,0) node[below] {$x$};
\draw[gray,>=stealth,->] (0,0.5) -- (0,0.8) node[left] {$v$};

\draw[orange] (-0.4,0.5) rectangle (0.4,-0.5) node[scale=0.9,right] {$B_0(\mathbf s,r)$};

\draw (0,-0.65)--(0,0.65) node[right] {$\{L^{\mathbf s}_1=0\}$};
\end{scope}

\begin{scope}[xshift=-1.85 cm]  
\draw[gray] (-0.5,0) -- (0.5,0) ;
\draw[gray] (0,-0.5) -- (0,0.5) ;

\filldraw[red,fill opacity=0.5] (-0.25,-0.65)--(0.45,0.65)--(0.25,0.65)--(-0.45,-0.65)-- cycle;

\draw (0,0) node {$\bullet$};

\fill[white] (-0.8,0.8) rectangle (0.8,0.5);
\fill[white] (-0.8,-0.8) rectangle (0.8,-0.5);

\draw (0.03,-0.05) node[scale=0.8] {$\mathbf s$};

\draw[gray] (-0.8,0) -- (-0.5,0) ;
\draw[gray] (0,-0.8) -- (0,-0.5) ;
\draw[gray,>=stealth,->] (0.5,0) -- (0.8,0) node[below] {$x$};
\draw[gray,>=stealth,->] (0,0.5) -- (0,0.8) node[left] {$v$};

\draw[orange] (-0.4,0.5) rectangle (0.4,-0.5) node[scale=0.9,right] {$B_0(\mathbf s,r)$};

\draw (-0.35,-0.65)--(0.35,0.65) node[right] {$\{L^{\mathbf s}_{\gamma_1+\nu}=0\}$};
\end{scope}
\end{tikzpicture}
\end{center}
Furthermore, we set 
\begin{align}\label{theta1}
\theta^{\mathbf m}_{\nu,h}=1 \quad  \text{on } \bigg(E(\mathbf m)+B(0, \varepsilon) \bigg)\Big\backslash \bigg(\bigsqcup_{\mathbf s\in\mathbf j(\mathbf m)} \Big(\bigcup_{\gamma \in [\gamma_1+\nu,1]}\{|U_\gamma^{\mathbf s}|\leq 2\delta \} \cap B_0(\mathbf s,r)\Big) \bigg)
\end{align}
with $\varepsilon=\varepsilon(r)>0$ to be fixed later and 
\begin{align}\label{theta0}
\theta^{\mathbf m}_{\nu,h}=0 \quad  \text{everywhere else.}
\end{align}
Note that $\theta^{\mathbf m}_{\nu,h}$ takes values in $[0,1]$ and that, thanks to \eqref{thetainte}, we also have
$$\theta^{\mathbf m}_{\nu,h}=1 \quad  \text{on } \Big( \bigcup_{\gamma \in [\gamma_1+\nu,1]}\{|U_\gamma^{\mathbf s}|\leq 2\delta \} \cap B_0(\mathbf s,r)\Big) \bigcap \Big( \bigcap_{\gamma \in [\gamma_1+\nu,1]} \{U_\gamma^{\mathbf s}\geq \delta \} \Big) $$
and
$$\theta^{\mathbf m}_{\nu,h}=0 \quad  \text{on } \Big( \bigcup_{\gamma \in [\gamma_1+\nu,1]}\{|U_\gamma^{\mathbf s}|\leq 2\delta \} \cap B_0(\mathbf s,r)\Big) \bigcap \Big( \bigcap_{\gamma \in [\gamma_1+\nu,1]} \{U_\gamma^{\mathbf s}\leq -\delta \} \Big) .$$
Denote $\Omega$ the CC of $\{W\leq \boldsymbol \sigma(\mathbf m)\}$ containing $\mathbf m$.
The CCs of $\{W\leq \boldsymbol \sigma(\mathbf m)\}$ are separated so for $\varepsilon>0$ small enough, there exists $\tilde \varepsilon>0$ such that 
$$\min\,\big\{W(x,v)\, ; \, \mathrm{dist} \big((x,v),\Omega\big)=\varepsilon\big\}=\boldsymbol \sigma(\mathbf m)+2 \tilde \varepsilon.$$
Thus the distance between $\{W \leq \boldsymbol \sigma(\mathbf m)+\tilde \varepsilon\}\cap \big(\Omega+B(0,\varepsilon)\big)$ and $\partial \big(\Omega+B(0,\varepsilon)\big)$ is positive and we can consider a cut-off function 
$$\chi_{\mathbf m}\in \mathcal C^{\infty}_c(\R^{2d},[0,1])$$
such that 
\begin{align}\label{chim}
\chi_{\mathbf m} =1 \text{ on } \{W \leq \boldsymbol \sigma(\mathbf m)+\tilde \varepsilon\}\cap \big(\Omega+B(0,\varepsilon)\big)
\qquad \text{ and } \qquad  
\mathrm{supp} \,\chi_{\mathbf m} \subset  \big(\Omega+B(0,\varepsilon)\big).
\end{align}
To sum up, we have the following picture:
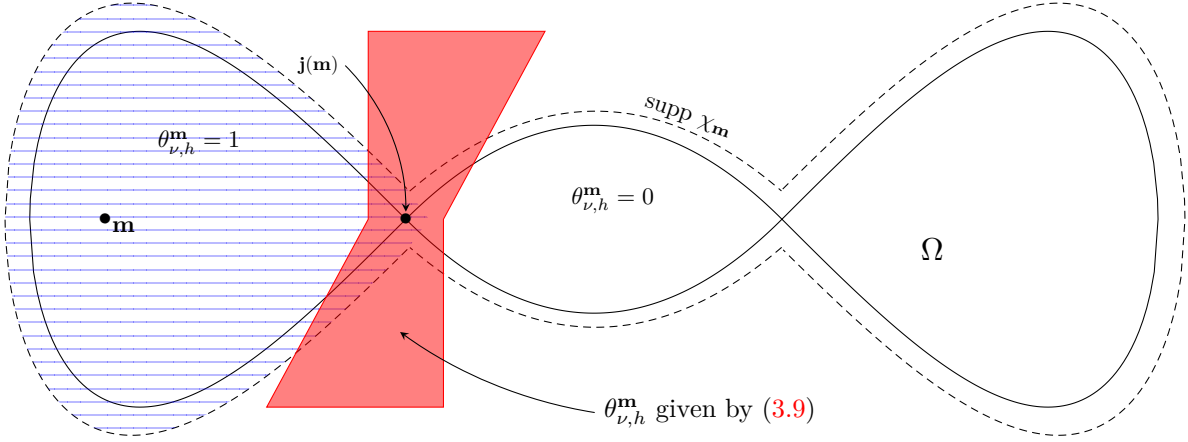
\begin{center}
\begin{tikzpicture}[scale=5]

\draw plot[samples=100,domain=1:2] (\x, {(\x-1)*sqrt(1-((\x-1))^2)}) 
		-- plot[samples=100,domain=2:1] (\x, {-(\x-1)*sqrt(1-((\x-1))^2)}) -- cycle ;
\draw plot[samples=100,domain=0:1] (\x, {\x*(1-\x)}) -- plot[samples=100,domain=1:0] (\x, {-\x*(1-\x)}) -- cycle ;
\draw plot[samples=100,domain=-1:0] (\x, {\x*sqrt(1-(\x)^2)}) -- plot[samples=100,domain=0:-1] (\x, {-\x*sqrt(1-(\x)^2)}) -- cycle ;

\begin{scope} [xshift=-6.5,scale=1.15]
\draw[densely dashed] plot[samples=100,domain=1.07:2] (\x, {(\x-1)*sqrt(1-((\x-1))^2)}) 
		-- plot[samples=100,domain=2:1.07] (\x, {-(\x-1)*sqrt(1-((\x-1))^2)}) ;
\end{scope}
\begin{scope} [xshift=-2,scale=1.15]
\draw[densely dashed] plot[samples=100,domain=0.07:0.93] (\x, {\x*(1-\x)});
\draw[densely dashed] plot[samples=100,domain=0.93:0.07] (\x, {-\x*(1-\x)});
\end{scope}
\begin{scope} [xshift=2.4,scale=1.15]
\draw[densely dashed] plot[samples=100,domain=-0.07:-1] (\x, {\x*sqrt(1-(\x)^2)}) -- plot[samples=100,domain=-1:-0.07] (\x, {-\x*sqrt(1-(\x)^2)});
\fill[pattern={Lines[angle=45,distance=5pt]},pattern color=blue,opacity=0.7] (-0.02,0) -- plot[samples=100,domain=-0.07:-1] (\x, {\x*sqrt(1-(\x)^2)}) -- plot[samples=100,domain=-1:-0.07] (\x, {-\x*sqrt(1-(\x)^2)}) -- cycle;
\end{scope}

\filldraw[red,fill opacity=0.5] (-0.1,0.5)--(-0.1,0)--(-0.37,-0.5)--(0.1,-0.5)--(0.1,0)--(0.37,0.5)-- cycle;

\draw (0,0) node {$\bullet$};
\draw (-0.8,0) node {$\bullet$};
\draw (-0.75,-0.02) node[scale=0.9] {$\mathbf m$};
\draw (-0.55,0.2) node[scale=0.9] {$\theta^{\mathbf m}_{\nu,h}=1$};
\draw (0.55,0.05) node[scale=0.9] {$\theta^{\mathbf m}_{\nu,h}=0$};
\draw (1.4,-0.08) node[scale=1.2] {$\Omega$};
\draw (0.75,0.27) node[scale=0.9,rotate=-19] {supp $\chi_{\mathbf m}$};

\begin{scope} [xshift=-3]
\draw[>=stealth,<-] (0.09,-0.3) arc ({(-0.7*pi) r}:{(-0.55*pi) r}:1.2) node[right] {$\theta^{\mathbf m}_{\nu,h}$ given by \eqref{thetainte}};
\end{scope}

\draw[>=stealth,<-] (0,0.02) arc ({0}:{(0.25*pi) r}:0.55) node[left,scale=0.72] {$\mathbf j(\mathbf m)$};
\end{tikzpicture}
\end{center}
The following Lemma will among other things help us discuss the regularity of $\theta_{\nu,h}^{\mathbf m}$.

\begin{lem}\label{bonmin}
Recall the notaion \eqref{hessiennes}. For all $\gamma\in (\gamma_1,1]$, we have
$$-\mathcal V_{\mathbf s}^{-1}L^{\mathbf s}_{\gamma,x}\cdot L^{\mathbf s}_{\gamma,x}-(L^{\mathbf s}_{\gamma,v})^2=1.$$
In particular, according to Lemma \ref{det}, $(\mathbf s,0)$ is a non degenerate minimum of $W+\frac12 (L_\gamma^{\mathbf s})^2$ and the associated hessian has determinant 
$$2^{-2d} \big|\det \mathcal V_{\mathbf s}\big|.$$
\end{lem}

\begin{proof}
It suffices to use \eqref{lv2} and \eqref{-Vlxlx}:
$$-\mathcal V_{\mathbf s}^{-1}L^{\mathbf s}_{\gamma,x}\cdot L^{\mathbf s}_{\gamma,x}-(L^{\mathbf s}_{\gamma,v})^2=2\,\frac{(1+\gamma)^2}{P(\gamma)}-\frac{(1-\gamma)^2}{P(\gamma)}=\frac{P(\gamma)}{P(\gamma)}=1.$$
For the computation of the determinant, it is sufficient to notice that, with the notation \eqref{hessiennes}, the hessian of $W+\frac12 (L_\gamma^{\mathbf s})^2$ at $(\mathbf s,0)$ is 
$$\mathcal W_{\mathbf s}+\begin{pmatrix}L^{\mathbf s}_{\gamma;x}\\L^{\mathbf s}_{\gamma;v}\end{pmatrix}\begin{pmatrix}L^{\mathbf s}_{\gamma;x}\\L^{\mathbf s}_{\gamma;v}\end{pmatrix}^t$$
and apply Lemma \ref{det}.
\end{proof}

\begin{prop}\label{thetalisse}
Up to changing 
the sign of $\ell^{\mathbf s}$, for all $\nu \in (0,|\gamma_1|)$, we can choose $\varepsilon>0$ and $\delta>0$ small enough so that the function $\theta^{\mathbf m}_{\nu,h}$ is smooth on the neighborhood of the support of $\chi_{\mathbf m}$ given by $\Omega+B(0,\varepsilon)$.
\end{prop}

\begin{proof}
Recall that by item \ref{jvide2} from Hypothesis \ref{jvide}, each $\ell^{\mathbf s}$ corresponds to a unique $\mathbf m  \in \mathtt U^{(0)}\backslash\{\underline{\mathbf m}\}$.
Let us first show that in $B_0(\mathbf s,r)$, we have
\begin{align}\label{antoine}
\bigcup_{\gamma \in [\gamma_1+\nu,1]}\{|U_\gamma^{\mathbf s}|\leq 2\delta \} = \big(\{U^{\mathbf s}_1\geq -2\delta\}\cap\{U^{\mathbf s}_{\gamma_1+\nu}\leq 2\delta\}\big) \cup \big(\{U^{\mathbf s}_1\leq 2\delta\}\cap\{U^{\mathbf s}_{\gamma_1+\nu}\geq -2\delta\}\big)
\end{align}
(so in particular, this set is closed).

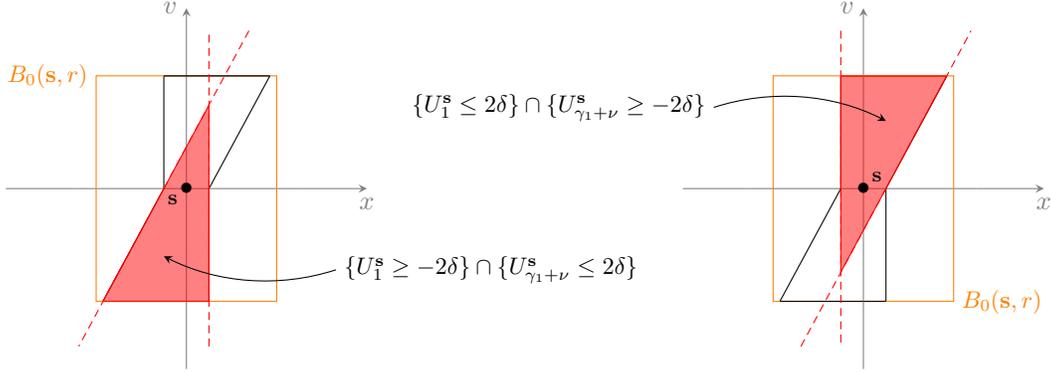
\begin{center}
\begin{tikzpicture}[scale=3]
\begin{scope}[xshift=1.5 cm]  
\draw[gray] (-0.5,0) -- (0.5,0) ;
\draw[gray] (0,-0.5) -- (0,0.5) ;
\begin{scope}[rotate=180]
\draw (-0.1,0.5)--(-0.1,0)--(-0.37,-0.5)--(0.1,-0.5)--(0.1,0)--(0.37,0.5)-- cycle;

\filldraw[red,fill opacity=0.5] (-0.37,-0.5)--(0.1,-0.5)--(0.1,0.37)-- cycle;

\draw (0,0) node {$\bullet$};

\fill[white] (-0.8,0.8) rectangle (0.8,0.5);
\fill[white] (-0.8,-0.8) rectangle (0.8,-0.5);

\draw[orange] (0.4,-0.5) rectangle (-0.4,0.5) node[scale=0.9,right] {$B_0(\mathbf s,r)$};

\draw (-0.06,-0.05) node[scale=0.8] {$\mathbf s$};

\draw (0.37,0.5)--(-0.1,0.5);
\draw[red] (-0.37,-0.5)--(0.1,-0.5);
\draw[red,densely dashed] (0.277,0.7)--(-0.477,-0.7);
\draw[red,densely dashed] (0.1,-0.7)--(0.1,0.7);

\draw[>=stealth,<-] (-0.1,-0.3) arc ({(-0.65*pi) r}:{(-0.4*pi) r}:1) node[scale=0.9,left] {$\{U^{\mathbf s}_1\leq 2\delta\}\cap\{U^{\mathbf s}_{\gamma_1+\nu}\geq -2\delta\}$};
\end{scope}

\draw[gray] (-0.8,0) -- (-0.5,0) ;
\draw[gray] (0,-0.8) -- (0,-0.5) ;
\draw[gray,>=stealth,->] (0.5,0) -- (0.8,0) node[below] {$x$};
\draw[gray,>=stealth,->] (0,0.5) -- (0,0.8) node[left] {$v$};

%
\end{scope}

\begin{scope}[xshift=-1.5 cm]  
\draw[gray] (-0.5,0) -- (0.5,0) ;
\draw[gray] (0,-0.5) -- (0,0.5) ;

\draw (-0.1,0.5)--(-0.1,0)--(-0.37,-0.5)--(0.1,-0.5)--(0.1,0)--(0.37,0.5)-- cycle;

\filldraw[red,fill opacity=0.5] (-0.37,-0.5)--(0.1,-0.5)--(0.1,0.37)-- cycle;

\draw (0,0) node {$\bullet$};

\fill[white] (-0.8,0.8) rectangle (0.8,0.5);
\fill[white] (-0.8,-0.8) rectangle (0.8,-0.5);

\draw[orange] (0.4,-0.5) rectangle (-0.4,0.5) node[scale=0.9,left] {$B_0(\mathbf s,r)$};

\draw (-0.06,-0.05) node[scale=0.8] {$\mathbf s$};

\draw[gray] (-0.8,0) -- (-0.5,0) ;
\draw[gray] (0,-0.8) -- (0,-0.5) ;
\draw[gray,>=stealth,->] (0.5,0) -- (0.8,0) node[below] {$x$};
\draw[gray,>=stealth,->] (0,0.5) -- (0,0.8) node[left] {$v$};

\draw (0.37,0.5)--(-0.1,0.5);
\draw[red] (-0.37,-0.5)--(0.1,-0.5);
\draw[red,densely dashed] (0.277,0.7)--(-0.477,-0.7);
\draw[red,densely dashed] (0.1,-0.7)--(0.1,0.7);

\draw[>=stealth,<-] (-0.1,-0.3) arc ({(-0.65*pi) r}:{(-0.4*pi) r}:1) node[scale=0.9,right] {$\{U^{\mathbf s}_1\geq -2\delta\}\cap\{U^{\mathbf s}_{\gamma_1+\nu}\leq 2\delta\}$};
\end{scope}

\end{tikzpicture}
\end{center}
Let $(x,v)\in \{U^{\mathbf s}_1\geq -2\delta\}\cap\{U^{\mathbf s}_{\gamma_1+\nu}\leq 2\delta\}$. If $U^{\mathbf s}_1(x,v)\leq 2\delta$, then $(x,v)\in \{|U_1^{\mathbf s}|\leq 2\delta \}$ and similarly, if $U^{\mathbf s}_{\gamma_1+\nu}(x,v)\geq -2\delta$, then $(x,v)\in \{|U_{\gamma_1+\nu}^{\mathbf s}|\leq 2\delta \}$. Now if $U^{\mathbf s}_1(x,v)> 2\delta$ and $U^{\mathbf s}_{\gamma_1+\nu}(x,v)< -2\delta$, by the intermediate value theorem, there exists $\gamma\in [\gamma_1+\nu,1]$ such that $U_\gamma^{\mathbf s}(x,v)=0$ so in particular $(x,v)\in \{|U_{\gamma}^{\mathbf s}|\leq 2\delta \}$.
Thus, we have shown that
\begin{align}\label{incl1}
\{U^{\mathbf s}_1\geq -2\delta\}\cap\{U^{\mathbf s}_{\gamma_1+\nu}\leq 2\delta\}\subseteq \bigcup_{\gamma \in [\gamma_1+\nu,1]}\{|U_\gamma^{\mathbf s}|\leq 2\delta \}
\end{align}
and clearly the same strategy of proof enables to show that
\begin{align}\label{incl2}
\{U^{\mathbf s}_1\leq 2\delta\}\cap\{U^{\mathbf s}_{\gamma_1+\nu}\geq -2\delta\}\subseteq \bigcup_{\gamma \in [\gamma_1+\nu,1]}\{|U_\gamma^{\mathbf s}|\leq 2\delta \}.
\end{align}
Conversely, let 
$$(x,v)\notin \big(\{U^{\mathbf s}_1\geq -2\delta\}\cap\{U^{\mathbf s}_{\gamma_1+\nu}\leq 2\delta\}\big) \cup \big(\{U^{\mathbf s}_1\leq 2\delta\}\cap\{U^{\mathbf s}_{\gamma_1+\nu}\geq -2\delta\}\big).$$
Since $\{U^{\mathbf s}_1< -2\delta\}\cap\{U^{\mathbf s}_{1}> 2\delta\}$ and $\{U^{\mathbf s}_{\gamma_1+\nu}< -2\delta\}\cap\{U^{\mathbf s}_{\gamma_1+\nu}> 2\delta\}$ are empty, we have 
\begin{align}\label{2delta}
(x,v)\in \{U^{\mathbf s}_1< -2\delta\}\cap\{U^{\mathbf s}_{\gamma_1+\nu}< -2\delta\} \quad \text{or} \quad (x,v)\in \{U^{\mathbf s}_{\gamma_1+\nu}> 2\delta\}\cap\{U^{\mathbf s}_1> 2\delta\}.
\end{align}
Besides, using \eqref{dyl} and \eqref{lv2}, one can check that the sign of $\partial_\gamma U^{\mathbf s}_\gamma(x,v)$ is given by 
\begin{align}\label{ugamma'}
\ell_x^{\mathbf s}\cdot \tilde x_{\mathbf s}-|\ell_x^{\mathbf s}|^2 \ell_v^{\mathbf s} \cdot v - \big( \ell_x^{\mathbf s}\cdot \tilde x_{\mathbf s}+|\ell_x^{\mathbf s}|^2 \ell_v^{\mathbf s} \cdot v\big) \gamma
\end{align}
which vanishes at most once in $(\gamma_1+\nu,1)$.
If it does not vanish in $(\gamma_1+\nu,1)$, then by monotonicity \eqref{2delta} implies that for any $\gamma\in [\gamma_1+\nu,1]$, we have $(x,v)\notin \{|U^{\mathbf s}_\gamma|\leq 2\delta\}$.
Now in the case where the expression from \eqref{ugamma'} vanishes at some point in $(\gamma_1+\nu,1)$, its values at $\gamma_1+\nu$ and $1$ have opposite signs, i.e
\begin{align}\label{memesigne}
|\ell_x^{\mathbf s}|^2 \ell_v^{\mathbf s} \cdot v \;\Big((1-\gamma_1-\nu)\ell_x^{\mathbf s}\cdot \tilde x_{\mathbf s} - |\ell_x^{\mathbf s}|^2 (1+\gamma_1+\nu) \ell_v^{\mathbf s} \cdot v\Big)>0.
\end{align}
When both factors from \eqref{memesigne} are positive, we have $\ell_x^{\mathbf s}\cdot \tilde x_{\mathbf s}>0$ so $U_1^{\mathbf s} (x,v)>0$ and it follows that $(x,v)\in \{U^{\mathbf s}_{\gamma_1+\nu}> 2\delta\}\cap\{U^{\mathbf s}_1> 2\delta\}$.
Moreover, we also have in that case that the minimum of $\gamma\mapsto U_\gamma^{\mathbf s}(x,v)$ on $[\gamma_1+\nu,1]$ is attained on the boundary of the interval since $\partial_\gamma U^{\mathbf s}_\gamma(x,v)|_{\gamma=1}<0$, so for any $\gamma\in [\gamma_1+\nu,1]$ it holds $(x,v)\in \{U^{\mathbf s}_\gamma> 2\delta\}$.
Here again, the same strategy of proof enables to show that if both factors from \eqref{memesigne} are negative, then for any $\gamma\in [\gamma_1+\nu,1]$, it holds $(x,v)\in \{U^{\mathbf s}_\gamma< -2\delta\}$.
Combined with \eqref{incl1} and \eqref{incl2}, this proves \eqref{antoine}.\\
From \eqref{thetainte}, \eqref{theta1}, \eqref{theta0} and \eqref{antoine}, we see that the only parts on which it is not clear that $\theta_{\nu,h}^\mathbf m$ is smooth are 
$$F_{1}=\bigsqcup_{\mathbf s \in \mathbf j(\mathbf m)}\Big(\{U_1^{\mathbf s}= 2\delta\}\cap \{U_{\gamma_1+\nu}^{\mathbf s}\geq 2\delta\} \cap  B_0(\mathbf s,r)\Big),$$ 
$$F_2=\bigsqcup_{\mathbf s \in \mathbf j(\mathbf m)}\Big(\{U_1^{\mathbf s}\geq 2\delta\}\cap \{U_{\gamma_1+\nu}^{\mathbf s}= 2\delta\} \cap  B_0(\mathbf s,r)\Big),$$
$$F_3=\bigsqcup_{\mathbf s \in \mathbf j(\mathbf m)}\Big(\{U_1^{\mathbf s}=-2\delta\}\cap \{U_{\gamma_1+\nu}^{\mathbf s}\leq- 2\delta\} \cap  B_0(\mathbf s,r)\Big),$$ 
$$F_4=\bigsqcup_{\mathbf s \in \mathbf j(\mathbf m)}\Big(\{U_1^{\mathbf s}\leq- 2\delta\}\cap \{U_{\gamma_1+\nu}^{\mathbf s}=- 2\delta\} \cap  B_0(\mathbf s,r)\Big),$$ 
$$F_5=\bigsqcup_{\mathbf s \in \mathbf j(\mathbf m)}\Big(\bigcup_{\gamma \in [\gamma_1+\nu,1]}\{|U_\gamma^{\mathbf s}|\leq 2\delta \}\cap \partial B_0(\mathbf s,r)\Big)$$
$$\qquad \text{ and } \quad F_6=\partial \bigg(E(\mathbf m)+B(0, \varepsilon) \bigg)\Big\backslash \bigg(\bigsqcup_{\mathbf s\in\mathbf j(\mathbf m)} \Big(\bigcup_{\gamma \in [\gamma_1+\nu,1]}\{|U_\gamma^{\mathbf s}|\leq 2\delta \} 
\cap B_0(\mathbf s,r)\Big) \bigg).$$
Note that \eqref{antoine} suggested to put $\{U_1^{\mathbf s}= 2\delta\}\cap \{U_{\gamma_1+\nu}^{\mathbf s}\geq -2\delta\} \cap  B_0(\mathbf s,r)$ in the definition of $F_1$, but we allowed ourselves to discard the part $\{U_1^{\mathbf s}= 2\delta\}\cap \{U_{\gamma_1+\nu}^{\mathbf s}\in [-2\delta,2\delta)\} \cap  B_0(\mathbf s,r)$ since it is included in the interior of $\{U_1^{\mathbf s}\geq -2\delta\}\cap \{U_{\gamma_1+\nu}^{\mathbf s} \leq  2\delta\} \cap  B_0(\mathbf s,r)$ (and we did similarly for $F_2$, $F_3$ and $F_4$).\\
Now, let $\mathbf s\in \mathbf j(\mathbf m)$ and $(\gamma,x,v) \in [\gamma_1+\nu,1]\times\overline{B_0(\mathbf s,r)}\backslash \{(\mathbf s,0)\}$ such that $U^{\mathbf s}_\gamma(x,v)=L_\gamma^{\mathbf s}(x,v)=0$.
Using Lemma \ref{bonmin}, we see that if $r>0$ is small enough,
\begin{align}\label{zerodel}
W(x,v)=W(x,v)+\frac12 L_\gamma^{\mathbf s}(x,v)^2>W(\mathbf s,0).
\end{align}
Hence, for all $\gamma\in[\gamma_1+\nu,1]$, the set $\{U_\gamma^{\mathbf s}=0\}\cap B_0(\mathbf s,r)$ is contained in $\{W\geq \boldsymbol \sigma(\mathbf m)\}$.
Assume by contradiction that for any $r>0$, the function $U_\gamma^{\mathbf s}$ takes both positive and negative values on $E(\mathbf m)\cap B_0(\mathbf s,r)$.
Then according to Lemma \ref{1.4}, the two CCs of $\mathcal O_r\cap \{W<\boldsymbol \sigma (\mathbf m)\}$ are both included in $E(\mathbf m)$ (the one on which $U_\gamma^{\mathbf s}>0$ and the one where $U_\gamma^{\mathbf s}<0$).
This is a contradiction with the fact that $\mathbf s\in \mathtt V^{(1)}$.
Therefore $U_\gamma^{\mathbf s}$ has a sign on $E(\mathbf m)\cap B_0(\mathbf s,r)$ and 
since it depends smoothly on $\gamma$ and cannot vanish on $E(\mathbf m)\cap B_0(\mathbf s,r)$, this sign does not depend on $\gamma$.
In particular, it is given by the sign of $U_0^{\mathbf s}$ on $E(\mathbf m)\cap B_0(\mathbf s,r)$ so taking $\ell^{\mathbf s}$ such that 
\begin{align}\label{x0}
\ell^{\mathbf s}\cdot (\phi_{\mathbf s}(x_0),v_0)>0
\end{align}
for some $(x_0,v_0)\in E(\mathbf m)\cap B_0(\mathbf s,r)$, we get that for each $\gamma\in [\gamma_1+\nu,1]$, the function $U_\gamma^{\mathbf s}$ is positive on $E(\mathbf m)\cap B_0(\mathbf s,r)$. 
We can then choose $\varepsilon(\delta)>0$ small enough so that 
\begin{align}\label{1.10}
\Big( \big(E(\mathbf m)+ B(0,\varepsilon) \big) \cap B_0(\mathbf s,r)\Big) \subseteq \big\{U_1^{\mathbf s} \geq -\delta\big\} \cap \big\{U_{\gamma_1+\nu}^{\mathbf s} \geq -\delta \big\}.
\end{align}
Similarly, if we denote $\Omega_{\mathbf s}$ the other CC of $\{W<\boldsymbol \sigma (\mathbf m)\}$ which contains $(\mathbf s,0)$ on its boundary, one can check that $\big( \phi^{-1}_{\mathbf s}(-\phi_{\mathbf s}(x_0)),-v_0 \big)\in \Omega_{\mathbf s}\cap B_0(\mathbf s,r)\cap \{U_0^{\mathbf s}<0\}$ 
where $(x_0,v_0)$ was introduced in \eqref{x0} so $U_\gamma^{\mathbf s}$ is negative 
on $\Omega_{\mathbf s}\cap B_0(\mathbf s,r)$ and 
\begin{align}\label{1.10bis}
\Big( \big(\Omega_{\mathbf s}+B(0, \varepsilon) \big) \cap B_0(\mathbf s,r)\Big) \subseteq \big\{U_1^{\mathbf s} \leq \delta\big\} \cap \big\{U_{\gamma_1+\nu}^{\mathbf s} \leq \delta \big\}.
\end{align}
Choosing once again $\varepsilon(r)$ small enough, we can even assume that 
\begin{align}\label{1.11}
\Big( \overline{E(\mathbf m)+B(0,\varepsilon)}\,  \cap \,   \overline{\Omega_{\mathbf s}+B(0,\varepsilon)} \Big) \subseteq \mathbf j^W(\mathbf m)+ B_0(0,r)
\end{align}
(see \cite{me}, Lemma 3.2 for more details).
We first prove that $\theta^{\mathbf m}_{\nu,h}$ is smooth on $F_1\cap (\Omega+B(0,\varepsilon))$: let $\mathbf s\in \mathbf j(\mathbf m)$ and $(x,v) \in B_0(\mathbf s,r)\cap \{U_1^{\mathbf s}= 2\delta\}\cap \{U_{\gamma_1+\nu}^{\mathbf s}\geq 2\delta\}\cap (\Omega+B(0,\varepsilon))$.
According to \eqref{1.10bis}, there exists a small ball $B$ centered in $(x,v)$ such that 
$$B\subset \Big(B_0(\mathbf s,r)\cap \{U_1^{\mathbf s}>\delta\}\cap \{U_{\gamma_1+\nu}^{\mathbf s}> \delta\}\cap \big(E(\mathbf m)+B(0,\varepsilon)\big)\Big).$$
Thus, according to \eqref{thetainte}, \eqref{theta1} and \eqref{antoine} with $\delta$ instead of $2\delta$, we have $\theta^{\mathbf m}_{\nu,h}=1$ on $B$ so $\theta^{\mathbf m}_{\nu,h}$ is smooth at $(x,v)$.
Obviously, the same goes for $F_2\cap (\Omega+B(0,\varepsilon))$ and similarly, for $(x,v) \in (F_3\cup F_4)\cap (\Omega+B(0,\varepsilon))$, we can show that $\theta^{\mathbf m}_{\nu,h}=0$ in a neighborhood of $(x,v)$.\\
Now we show that $F_5$ does not meet $\Omega+B(0,\varepsilon)$.
Recall that $\Omega$ denotes the CC of $\{W\leq \boldsymbol \sigma(\mathbf m)\}$ containing $\mathbf m$.
For $\mathbf s\in \mathbf j(\mathbf m)$, we can deduce from \eqref{zerodel} that if $(\gamma,x,v)\in [\gamma_1+\nu,1]\times\partial B_0(\mathbf s,r)$ is such that $U_\gamma^{\mathbf s}(x,v)=0$, then $(x,v) \notin  \Omega$.
Hence $(\gamma,x,v)\mapsto|U_\gamma^{\mathbf s}(x,v)|$ must attain a positive minimum on $[\gamma_1+\nu,1]\times(\partial B_0(\mathbf s,r)\cap  \Omega)$, so we can choose $\delta(r,\nu)>0$ independent of $\gamma$ such that for all $\gamma\in[\gamma_1+\nu,1]$, the set $\partial B_0(\mathbf s,r)\cap \{|U_\gamma^{\mathbf s}|\leq 2 \delta\}$ does not intersect $ \Omega$.
It follows that we can choose $\varepsilon (\delta)>0$ such that 
\begin{align*}
F_5 \subseteq \big( \R^{2d} \backslash \overline{\Omega+B(0,\varepsilon)} \big).
\end{align*}
It only remains to prove that, as for $F_5$, the set $F_6$ does not meet $\Omega+B(0,\varepsilon)$.
If $(x,v)\in F_6 \cap B_0(\mathbf s,r)$, \eqref{1.10} and \eqref{antoine} imply that $(x,v) \in \{U^{\mathbf s}_1 \geq 2\delta\}\cap \{U^{\mathbf s}_{\gamma_1+\nu} \geq 2\delta\}$ so using \eqref{1.10bis}, we see that $(x,v)$ is outside $\Omega_{\mathbf s}+B(0,\varepsilon)$.
Since it is not in $(E(\mathbf m)+B(0,\varepsilon))$ either, it is outside $\Omega+B(0,\varepsilon)$. 
Now if $(x,v)\in F_6 \backslash \big( \mathbf j^W(\mathbf m)+B_0(0,r)\big)$, \eqref{1.11} implies that $(x,v)$ is outside $\cup_{\mathbf j(\mathbf m)}(\Omega_{\mathbf s}+B(0,\varepsilon))$ so it is also outside $\Omega+B(0,\varepsilon)$ for $\varepsilon$ small enough and the proof is complete.
\end{proof}
\hip
From now on, we fix the sign of $\ell^{\mathbf s}$ as well as $\varepsilon>0$ and $\delta>0$ such that the conclusion of Proposition \ref{thetalisse} holds true.
In particular, even though we do not make it appear in the notations, the functions $\chi_{\mathbf m}$ and $\zeta$ now depend on $\nu$.
Finally, we denote 
\begin{align}\label{wm}
W^{\mathbf m}(x,v)=W(x,v)-V(\mathbf m)/2
\end{align}
and it is clear from \eqref{chim} that
\begin{align}\label{suppchi}
W^{\mathbf m}\geq S(\mathbf m)+\tilde \varepsilon \qquad \text{on supp }\nabla \chi_{\mathbf m}.
\end{align}
Our quasimodes will be the $L^2$-renormalizations of the functions 
\begin{align}\label{quasim}
f_{\nu,h}^{\mathbf m}(x,v)=\chi_{\mathbf m}(x,v)\theta^{\mathbf m}_{\nu,h}(x,v)\e^{-W^{\mathbf m}(x,v)/h}\quad ; \quad  \mathbf m \in \mathtt U^{(0)}\backslash\{\underline{\mathbf m}\}
\end{align}
and for $\mathbf m=\underline{\mathbf m}$, 
$$f_{\underline{\mathbf m},h}(x,v)=\e^{-W^{\underline{\mathbf m}}(x,v)/h}\in \mathrm{Ker}\, P_h.$$
Note that for $\mathbf m\neq \underline{\mathbf m}$, we have $f_{\nu,h}^{\mathbf m}\in \mathcal C^\infty_c(\R^{2d})$ thanks to Proposition \ref{thetalisse} and
\begin{align}\label{suppf}
\mathrm{supp}\, f_{\nu,h}^{\mathbf m} \subseteq E(\mathbf m)+B(0,\varepsilon')
\end{align}
where $\varepsilon'=\max(\varepsilon,r)$.

\subsection{Action of the operator $P_h$}

Let us fix $\mathbf m \in \mathtt U^{(0)}\backslash\{\underline{\mathbf m}\}$.
For $\gamma \in (\gamma_1,1]$, we will denote
\begin{align}\label{ytilde}
\widetilde W^{\mathbf m}_{\gamma}(x,v)=W^{\mathbf m}(x,v)+\frac12 \sum_{\mathbf s \in \mathbf j(\mathbf m)}L^{\mathbf s}_{\gamma}(x,v)^2.
\end{align}
For $\mathbf s\in \mathbf j(\mathbf m)$ and $x \in B(\mathbf s, \tilde r)$ we also denote
\begin{align}\label{thetatilde}
\tilde \theta^{\mathbf s}_{\gamma,h} (x,v)=\int_0^{L^{\mathbf s}_\gamma(x,v)}\e^{-\frac{s^2}{2h}}\D s.
\end{align}
We now have to compute $P_hf_{\nu,h}^{\mathbf m}$.
We will see fairly easily thanks to \eqref{nablatheta} that $X_0^h$ applied to $f_{\nu,h}^{\mathbf m}$ will yield a superposition of the exponentials 
\begin{align}\label{phasesx0}
\Big(\e^{-\widetilde W^{\mathbf m}_{\gamma}/h}\Big)_{\gamma\in [\gamma_1+\nu,1]}.
\end{align}
In view of \eqref{thetainte}, we see that the computation of $Q_hf_{\nu,h}^{\mathbf m}$ will essentially boil down to the one of $Q_h(\tilde \theta^{\mathbf s}_{\gamma,h}\e^{-W^{\mathbf m}/h})$ which we are already able to do thanks to Lemma \ref{actionq}:
\begin{align}
Q_h(\tilde \theta^{\mathbf s}_{\gamma,h}\e^{-W^{\mathbf m}/h})=-h \int_0^1 \partial_y \mathscr L^{\mathbf s}(\gamma,y)\,\exp\Big[-\frac 1h\Big(W^{\mathbf m}(x,v)+\frac12 \big[\mathscr L^{\mathbf s}(\gamma,y)\cdot (\tilde x_{\mathbf s},v)\big]^2\Big)\Big]\,\D y\,\cdot \begin{pmatrix}\tilde x_{\mathbf s}\\v\end{pmatrix}
\end{align}
where $\mathscr L^{\mathbf s}(\gamma,y)$ stands for the vector
\begin{align}\label{Ltilde}
\begin{pmatrix}
\frac{1+y}{\big(4|L_{\gamma,v}^{\mathbf s}|^2y+(y+1)^2\big)^{1/2}}\,L_{\gamma,x}^{\mathbf s} & ; & \frac{1-y}{\big(4|L_{\gamma,v}^{\mathbf s}|^2y+(y+1)^2\big)^{1/2}}\,L_{\gamma,v}^{\mathbf s}
\end{pmatrix}.
\end{align}
Here we disregarded the fact that the linear form $L_\gamma$ is twisted in $x$ as $Q_h$ only acts in $v$.
Our concern is now to see whether the functions 
$$\bigg(\exp\Big[-\frac 1h\Big(W^{\mathbf m}(x,v)+\frac12 \big[\mathscr L^{\mathbf s}(\gamma,y)\cdot (\tilde x_{\mathbf s},v)\big]^2\Big)\Big]\bigg)_{\gamma\in [\gamma_1+\nu,1]\,;\, y\in [0,1]}$$
belong to the family \eqref{phasesx0} as we hoped for some compensations between $X_0^hf_{\nu,h}^{\mathbf m}$ and $Q_hf_{\nu,h}^{\mathbf m}$.
It appears to be the case as, denoting for $\gamma\in (\gamma_1,1]$ and $y\in [0,1]$
\begin{align}\label{Gamma}
\Gamma_\gamma(y)=\frac{y+\gamma}{1+y\gamma},
\end{align}
an easy computation shows that
\begin{align}\label{Lcircgamma}
\mathscr L^{\mathbf s}(\gamma,y)=\big(L_{\Gamma_\gamma(y),x}^{\mathbf s} \, ; \, L_{\Gamma_\gamma(y),v}^{\mathbf s}\big).
\end{align}
We sum up the above discussion in the following updated version of Lemma \ref{actionq}.

\begin{lem}\label{actionq2}
With the notations \eqref{thetatilde}, \eqref{Gamma} and \eqref{ytilde}, we have
$$h\, \mathrm{Op}_h(g)\Big( \e^{-\frac{\widetilde W^{\mathbf m}_{\gamma}}{h}}  L_{\gamma,v}^{\mathbf s} \Big)=Q_h(\tilde \theta^{\mathbf s}_{\gamma,h}\e^{-W^{\mathbf m}/h})(x,v)=
-h\int_0^1 \partial_y(L_{\Gamma_\gamma(y)})\,\e^{-\frac{\widetilde W^{\mathbf m}_{\Gamma_\gamma(y)}}{h}}\,\D y\,\cdot \begin{pmatrix}\tilde x_{\mathbf s}\\v\end{pmatrix}.$$
Moreover,
\begin{align}\label{opm2}
\text{ }\qquad \mathrm{Op}_h(m_{y,h} \, \mathrm{Id})\circ b_h\Big(\tilde \theta^{\mathbf s}_{\gamma,h}\e^{-\frac{W^{\mathbf m}(x,v)}{h}}\Big)&=2h(2\pi h)^{-d/2}\e^{-\frac{V(x)-V(\mathbf m)}{2h}} \frac{(y+1)^{d-2}}{(4y)^{\frac d2}}\\
		&\qquad  \quad\times \int_{v'\in \R^d} \e^{-\frac 1h \big(\frac{v'^2}{4}+\frac y8 (v+v')^2+\frac{(v-v')^2}{8y}+\frac12 L^{\mathbf s}_\gamma(x,v')^2\big)} \,\D v'  L_{\gamma,v}^{\mathbf s}.
\end{align}
\end{lem}
\hip
We are now in position to give a precise computation of $P_hf_{\nu,h}^{\mathbf m}$.

\begin{prop}\label{phf}
Let $f_{\nu,h}^{\mathbf m}$ be the quasimode defined in \eqref{quasim} and recall the notations \eqref{xtilde} and \eqref{ytilde}.
There exist some functions $R^{\mathbf m}_{\nu, h}$ and $(\om^{\mathbf m}_{\nu,z})_{z\in [\gamma_1+\nu,1]}$ in $L^2(\R^{2d})$ such that 
\begin{enumerate}
\item The function $P_hf_{\nu,h}^{\mathbf m}-R^{\mathbf m}_{\nu, h}$ is supported in $\mathbf j^W(\mathbf m)+B_0(0,r)$.
\item The function $R^{\mathbf m}_{\nu, h}$ is $O_{\nu,L^2}\Big(h^{\frac{3+d}{2}}\e^{-\frac{S(\mathbf m)}{h}}\Big)$.
\item For $(x,v)\in \mathbf j^W(\mathbf m)+B_0(0,r)$, one has
	$$\big(P_hf_{\nu,h}^{\mathbf m}-R^{\mathbf m}_{\nu, h}\big)(x,v)=\Big(\frac{h}{2\pi}\Big)^{1/2} \int_{\gamma_1+\nu}^1\om^{\mathbf m}_{\nu,z}(x,v)\,\exp \Big[-\frac 1h \widetilde W^{\mathbf m}_z(x,v)\Big]\D z$$
where, using the notation \eqref{hessiennes}, we have the expression
\begin{align*}
\om^{\mathbf m}_{\nu,z}(x,v)=\sum_{\mathbf s\in \mathbf j(\mathbf m)} \bigg[k_\nu^{\mathbf s}(z)\begin{pmatrix}0 & -\mathcal V_{\mathbf s}\\ \mathrm{Id} & 0 \end{pmatrix}\begin{pmatrix}L^{\mathbf s}_{z;x}\\L^{\mathbf s}_{z;v}\end{pmatrix}
-\int_{\gamma_1+\nu}^z k_\nu^{\mathbf s}(\gamma)\,\D \gamma \begin{pmatrix}\partial_z L^{\mathbf s}_{z;x}\\\partial_z L^{\mathbf s}_{z;v}\end{pmatrix}\bigg]
\cdot \begin{pmatrix}\tilde x_{\mathbf s}\\v\end{pmatrix}.
\end{align*}
\end{enumerate}
\end{prop}

\begin{proof}
In order to lighten the notations, we will drop some of the exponents and indexes $\mathbf m$, $\mathbf s$, $\nu$ and $h$ in the proof.
We know that $\theta$ is smooth on the support of $\chi$ and since $\theta$ is constant outside of $\mathbf j^W(\mathbf m)+B_0(0,r)$, we have
\begin{align}\label{nablatheta}
\nabla \theta=\frac{1}{2}\sum_{\mathbf s\in \mathbf j(\mathbf m)} (A_h^{\mathbf s})^{-1} \int_{\gamma_1+\nu}^{1}k^{\mathbf s}(\gamma)\zeta(U^{\mathbf s}_\gamma)\e^{-(L_\gamma^{\mathbf s})^2/2h}\,\nabla L_\gamma^{\mathbf s} \, \1_{B_0(\mathbf s,r)} \,\D \gamma .
\end{align}
Using Corollary \ref{opgb}, we can then begin by computing
\begin{align}
\qquad \quad Q_h(f)&=h\mathrm{Op}_h(g)\big((\partial_v \theta) \chi \e^{-W^{\mathbf m}/h}+(\partial_v\chi)\theta \e^{-W^{\mathbf m}/h}\big)\label{qfdecompo}\\
			&=\frac h2   \sum_{\mathbf s\in \mathbf j(\mathbf m)} (A_h^{\mathbf s})^{-1} \int_{\gamma_1+\nu}^{1}k^{\mathbf s}(\gamma)  \mathrm{Op}_h(g)\Big( \chi \zeta(U^{\mathbf s}_\gamma) \, \e^{-\frac{\widetilde W^{\mathbf m}_{\gamma}}{h}} \1_{B_0(\mathbf s,r)}  L_{\gamma,v}^{\mathbf s} \Big) \,\D \gamma+O_\nu\Big(h\e^{-\frac{S(\mathbf m)+\tilde \varepsilon}{h}}\Big) \label{qf}
\end{align}
as $\chi$ now depends on $\nu$, where we used \eqref{suppchi} as well as the fact that $\mathrm{Op}_h(g)$ is bounded uniformly in $h$ since $g \in S^{1/2}(\langle (v,\eta) \rangle^{-1})$.
Now, since $\chi \zeta(U^{\mathbf s}_\gamma)-1=O_\nu\big((x-\mathbf s,v)^2\big)$, we have
thanks to Lemma \ref{bonmin} and by a standard Laplace method (see Proposition \ref{lap}) that
\begin{align}\label{chizeta-1}
\big(\chi \zeta(U^{\mathbf s}_\gamma) -1\big) \e^{-\frac{\widetilde W^{\mathbf m}_{\gamma}}{h}} \1_{B_0(\mathbf s,r)}  \nabla L_{\gamma}^{\mathbf s}=O_\nu\Big(h^{1+\frac d2}\e^{-\frac{S(\mathbf m)}{h}}\Big).
\end{align}
Hence, still by the boundedness of $\mathrm{Op}_h(g)$, we get that
\begin{align}\label{chizeta}
\mathrm{Op}_h(g)\Big( \chi \zeta(U^{\mathbf s}_\gamma) \, \e^{-\frac{\widetilde W^{\mathbf m}_{\gamma}}{h}} \1_{B_0(\mathbf s,r)}  L_{\gamma,v}^{\mathbf s} \Big)=\mathrm{Op}_h(g)\Big( \e^{-\frac{\widetilde W^{\mathbf m}_{\gamma}}{h}} \1_{B_0(\mathbf s,r)}  L_{\gamma,v}^{\mathbf s} \Big)+O_\nu\Big(h^{1+\frac d2}\e^{-\frac{S(\mathbf m)}{h}}\Big).
\end{align}
In the same spirit, we can write
$$\1_{B_0(\mathbf s,r)}  L_{\gamma,v}^{\mathbf s}=\1_{|x-\mathbf s|<\tilde r}(\1_{|v|<r}-1+1)  L_{\gamma,v}^{\mathbf s}=\1_{|x-\mathbf s|<\tilde r}\, L_{\gamma,v}^{\mathbf s} +\rho_\gamma$$
with $\rho_\gamma$ supported in $\{(x,v); \, |x-\mathbf s|<\tilde r \text{ and }|v|\geq r\}$ and such that $\|\rho_\gamma\|_{\infty}\leq C_\nu$, so using the boudedness of $\mathrm{Op}_h(g)$ again and the fact that it is local in the variable $x$, as well as \eqref{rtilde}, we get
\begin{align}\label{indic}
\mathrm{Op}_h(g)\Big( \e^{-\frac{\widetilde W^{\mathbf m}_{\gamma}}{h}} \1_{B_0(\mathbf s,r)}  L_{\gamma,v}^{\mathbf s} \Big)=\mathrm{Op}_h(g)\Big( \e^{-\frac{\widetilde W^{\mathbf m}_{\gamma}}{h}} L_{\gamma,v}^{\mathbf s} \Big)\1_{|x-\mathbf s|<\tilde r}+O_\nu\Big(h^{1+\frac d2}\e^{-\frac{S(\mathbf m)}{h}}\Big).
\end{align}
Hence, putting \eqref{chizeta} and \eqref{indic} together and using \eqref{approxa}, we get that \eqref{qf} becomes
\begin{align}\label{bhfsimple}
Q_h(f)=\Big(\frac {h}{2\pi}\Big)^{1/2}  \sum_{\mathbf s\in \mathbf j(\mathbf m)}  \int_{\gamma_1+\nu}^{1} k^{\mathbf s}(\gamma)  \mathrm{Op}_h(g)\Big( \e^{-\frac{\widetilde W^{\mathbf m}_{\gamma}}{h}}  L_{\gamma,v}^{\mathbf s} \Big) \D \gamma\, \1_{|x-\mathbf s|<\tilde r}+O_\nu\Big(h^{\frac{3+d}{2}}\e^{-\frac{S(\mathbf m)}{h}}\Big)
\end{align}
which further gives
\begin{align}\label{expltilde}
Q_h&(f)+O_\nu\Big(h^{\frac{3+d}{2}}\e^{-\frac{S(\mathbf m)}{h}}\Big)=\\
			&-\Big(\frac {h}{2\pi}\Big)^{1/2}  \sum_{\mathbf s\in \mathbf j(\mathbf m)}  \int_{\gamma_1+\nu}^{1} k^{\mathbf s}(\gamma)  \int_0^1 \partial_y(L_{\Gamma_\gamma(y)})\,\exp\Big[-\frac 1h \widetilde W^{\mathbf m}_{\Gamma_\gamma(y)}\Big]\,\D y\,\cdot \begin{pmatrix}\tilde x_{\mathbf s}\\v\end{pmatrix}       \D \gamma\,\1_{|x-\mathbf s|<\tilde r} \nonumber
\end{align}
thanks to Lemma \ref{actionq2}.
By the change of variable $z=\Gamma_\gamma(y)$, the integral in $y$ from \eqref{expltilde} becomes
$$\int_\gamma^1 \partial_z( L^{\mathbf s}_{z})\,\exp\Big[- \widetilde W^{\mathbf m}_{z}(x,v)/h
\Big]\,\D z.$$
Therefore, switching the order of integration and using \eqref{rtilde} again, \eqref{expltilde} yields that up to a \\$O_{\nu,L^2}(h^{\frac{3+d}{2}}\e^{-S(\mathbf m)/h})$, the function $Q_h(f)$ satisfies
\begin{align}\label{qffinal}
Q_h(f)(x,v)&=-\Big(\frac {h}{2\pi}\Big)^{1/2}  \sum_{\mathbf s\in \mathbf j(\mathbf m)} \int_{\gamma_1+\nu}^1  \int_{\gamma_1+\nu}^{z} k^{\mathbf s}(\gamma) \, \D \gamma \, \partial_z( L^{\mathbf s}_{z})\,  \cdot \begin{pmatrix}\tilde x_{\mathbf s}\\v\end{pmatrix}      \e^{-\frac{\widetilde W^{\mathbf m}_{z}(x,v)}{h}}
 \D z \,\1_{B_0(\mathbf s, r)}(x,v)
\end{align}
Now the computation for the transport term is easier: according to \eqref{nablatheta}, we have
\begin{align}\label{x0f}
X_0^hf&=h\begin{pmatrix}
v \\
-\partial_x V 
\end{pmatrix} \cdot \nabla f \nonumber\\
			&=h\begin{pmatrix}
v \\
-\partial_x V 
\end{pmatrix} \cdot \nabla \theta \,\chi \e^{-W^{\mathbf m}/h}+h\begin{pmatrix}
v \\
-\partial_x V 
\end{pmatrix} \cdot \nabla \chi \,\theta \e^{-W^{\mathbf m}/h}\\
			&=\frac h2 \chi  \sum_{\mathbf s\in \mathbf j(\mathbf m)} (A_h^{\mathbf s})^{-1} \int_{\gamma_1+\nu}^{1}k^{\mathbf s}(z)\zeta(U^{\mathbf s}_z)\,\begin{pmatrix}
v \\
-\partial_x V 
\end{pmatrix}\cdot \nabla L_z^{\mathbf s}\, \e^{-\frac{\widetilde W^{\mathbf m}_{z}}{h}} \1_{B_0(\mathbf s,r)} \,\D z  +O_\nu\Big(h\e^{-\frac{S(\mathbf m)+\tilde \varepsilon}{h}}\Big) \nonumber
\end{align}
thanks to \eqref{suppchi}.
Here again, we can use \eqref{approxa} and \eqref{chizeta-1} to get 
\begin{align}
X_0^hf=\Big(\frac {h}{2\pi}\Big)^{1/2}  \sum_{\mathbf s\in \mathbf j(\mathbf m)}  \int_{\gamma_1+\nu}^{1}k^{\mathbf s}(z)\,\begin{pmatrix}
v \\
-\partial_x V 
\end{pmatrix}\cdot \nabla L_z^{\mathbf s}\, \e^{-\frac{\widetilde W^{\mathbf m}_{z}}{h}} \1_{B_0(\mathbf s,r)} \,\D z  +O_\nu\Big(h^{\frac{3+d}{2}}\e^{-\frac{S(\mathbf m)}{h}}\Big). 
\end{align}
Recalling that the differential of $\phi_{\mathbf s}$ at $\mathbf s$ is the identity, the last step consists in using \eqref{phi} to write
\begin{align}
\begin{pmatrix}
v \\
-\partial_x V 
\end{pmatrix}\cdot \nabla L_z^{\mathbf s}&=\begin{pmatrix}
0&\mathrm{Id}\\
-\mathcal V_{\mathbf s} & 0
\end{pmatrix}
\begin{pmatrix}
\tilde x_{\mathbf s}\\
v
\end{pmatrix}
\cdot
\begin{pmatrix}
 L_{z,x}^{\mathbf s}\\
L_{z,v}^{\mathbf s}
\end{pmatrix}
+O_\nu\Big((\tilde x_{\mathbf s},v)^2\Big)\\
	&=\begin{pmatrix}
0&-\mathcal V_{\mathbf s}\\
\mathrm{Id} & 0
\end{pmatrix}
\begin{pmatrix}
L_{z,x}^{\mathbf s}\\
L_{z,v}^{\mathbf s}
\end{pmatrix}
\cdot
\begin{pmatrix}
\tilde x_{\mathbf s}\\
v
\end{pmatrix}
+O_\nu\Big((\tilde x_{\mathbf s},v)^2\Big)
\end{align}
and the same argument that we used to establish \eqref{chizeta-1} yields
that up to a $O_{\nu,L^2}(h^{\frac{3+d}{2}}\e^{-S(\mathbf m)/h})$, the function $X_0^hf$ satisfies
\begin{align}\label{x0ffinal}
\text{ }\qquad X_0^hf (x,v)=\Big(\frac {h}{2\pi}\Big)^{1/2}  \sum_{\mathbf s\in \mathbf j(\mathbf m)}  \int_{\gamma_1+\nu}^{1}k^{\mathbf s}(z)
\begin{pmatrix}
0&-\mathcal V_{\mathbf s}\\
\mathrm{Id} & 0
\end{pmatrix}
\begin{pmatrix}
L_{z,x}^{\mathbf s}\\
L_{z,v}^{\mathbf s}
\end{pmatrix}
\cdot
\begin{pmatrix}
\tilde x_{\mathbf s}\\
v
\end{pmatrix}
\e^{-\frac{\widetilde W^{\mathbf m}_{z}(x,v)}{h}}  \,\D z  \1_{B_0(\mathbf s,r)}(x,v).
\end{align}
The conclusion follows from \eqref{qffinal} and \eqref{x0ffinal}.
\end{proof}

\begin{rema}\label{Ph*f}
Since $P_h^*=-X_0^h+Q_h$, it is clear from \eqref{qffinal} and \eqref{x0ffinal} that 
$$P_h^*f^{\mathbf m}_{\nu,h}=\Big(\frac{h}{2\pi}\Big)^{1/2} \int_{\gamma_1+\nu}^1\overset{*}{\om}{}^{\mathbf m}_{\nu,z}(x,v)\,\exp \Big[-\frac 1h \widetilde W^{\mathbf m}_z(x,v)\Big]\D z+O_{\nu,L^2}\Big(h^{\frac{3+d}{2}}\e^{-\frac{S(\mathbf m)}{h}}\Big)$$
with 
\begin{align*}
\overset{*}{\om}{}^{\mathbf m}_{\nu,z}(x,v)=-\sum_{\mathbf s\in \mathbf j(\mathbf m)} \bigg[k_\nu^{\mathbf s}(z)\begin{pmatrix}0 & -\mathcal V_{\mathbf s}\\ \mathrm{Id} & 0 \end{pmatrix}\begin{pmatrix}L^{\mathbf s}_{z;x}\\L^{\mathbf s}_{z;v}\end{pmatrix}
+\int_{\gamma_1+\nu}^z k_\nu^{\mathbf s}(\gamma)\,\D \gamma \begin{pmatrix}\partial_z L^{\mathbf s}_{z;x}\\\partial_z L^{\mathbf s}_{z;v}\end{pmatrix}\bigg]
\cdot \begin{pmatrix}\tilde x_{\mathbf s}\\v\end{pmatrix}\1_{\mathbf j^W(\mathbf m)+B_0(0,r)} (x,v).
\end{align*}
\end{rema}

\hop
\subsection{Choices of $\ell$ and $k$}

Following the steps from \cite{me,BonyLPMichel}, we would like in view of Proposition \ref{phf} to find $(\ell^{\mathbf s})_{\mathbf s\in \mathbf j(\mathbf m)}\subset \R^{2d}$ satisfying \eqref{lv2} and \eqref{-Vlxlx} as well as some probability densities $(k_\nu^{\mathbf s})_{\mathbf s\in \mathbf j(\mathbf m)}$ on $[\gamma_1+\nu, 1]$ for which the leading term of $P_hf_{\nu,h}^{\mathbf m}$ vanishes, i.e such that 
\begin{align}\label{eikon}
k_\nu^{\mathbf s}(z)\begin{pmatrix}0 & -\mathcal V_{\mathbf s}\\ \mathrm{Id} & 0 \end{pmatrix}\begin{pmatrix}L^{\mathbf s}_{z;x}\\L^{\mathbf s}_{z;v}\end{pmatrix}
-\int_{\gamma_1+\nu}^z k_\nu^{\mathbf s}(\gamma)\,\D \gamma \begin{pmatrix}\partial_z L^{\mathbf s}_{z;x}\\\partial_z L^{\mathbf s}_{z;v}\end{pmatrix}=0, \quad \; \forall \mathbf s\in \mathbf j(\mathbf m),\;  \forall z \in [\gamma_1+\nu, 1].
\end{align}
\hip
As it will be more convenient to handle than the function $k_\nu^{\mathbf s}$, let us introduce the cumulative distribution function (CDF) on $[\gamma_1+\nu,1]$ associated to $k_\nu^{\mathbf s}$ :
\begin{align}\label{K}
K^{\mathbf s}_\nu (z)=\int_{\gamma_1+\nu}^z k_\nu^{\mathbf s}(\gamma)\,\D \gamma .
\end{align}
\begin{lem}\label{sisol}
Recall the notations \eqref{hessiennes}-\eqref{tau}.
If $(\ell^{\mathbf s})_{\mathbf s\in \mathbf j(\mathbf m)}$ is a family of vectors satisfying \eqref{lv2} and $(k_\nu^{\mathbf s})_{\mathbf s\in \mathbf j(\mathbf m)}$ is a family of probability densities on $[\gamma_1+\nu, 1]$ for which \eqref{eikon} holds true, then
\begin{align}\label{choixl}
\mathcal V_{\mathbf s}\ell_v^{\mathbf s}=\tau_{\mathbf s} \ell_v^{\mathbf s}\qquad ;\qquad 
\ell_x^{\mathbf s}=-\sqrt{2|\tau_{\mathbf s}|}\ell_v^{\mathbf s}
\end{align}
(in particular, $\ell_x^{\mathbf s}$ satisfies \eqref{-Vlxlx}) and the function $K_\nu^{\mathbf s}$ defined in \eqref{K} is a CDF on $[\gamma_1+\nu,1]$ satisfying the ODE
\begin{align}\label{ode}
(K_\nu^{\mathbf s})'(z)-\frac{2\sqrt 2}{\sqrt{|\tau_{\mathbf s}|}P(z)}K_\nu^{\mathbf s}(z)=0.
\end{align}
\end{lem}

\begin{proof}
Let $(\ell^{\mathbf s})_{\mathbf s\in \mathbf j(\mathbf m)}$ and $(k_\nu^{\mathbf s})_{\mathbf s\in \mathbf j(\mathbf m)}$ satisfying the hypotheses of the lemma.
According to \eqref{dyl}, \eqref{lv2} and \eqref{eikon}, we have
\begin{align}\label{E12}
-k^{\mathbf s}_\nu(z)\mathcal V_{\mathbf s}\ell_v^{\mathbf s}+2\frac{K^{\mathbf s}_\nu (z)}{P(z)} \ell_x^{\mathbf s}=0\qquad \text{and } \qquad k^{\mathbf s}_\nu(z)\ell_x^{\mathbf s}+4\frac{K^{\mathbf s}_\nu (z)}{P(z)} \ell_v^{\mathbf s}=0
\end{align}
from which we deduce that there exists $\sigma_{\mathbf s}<0$ such that $\ell_x^{\mathbf s}=\sigma_{\mathbf s} \ell_v^{\mathbf s}$ and consequently, that $\ell_v^{\mathbf s}$ is an eigenvector of $\mathcal V_{\mathbf s}$ associated to its negative eigenvalue $\tau_{\mathbf s}$.
Plugging these informations in \eqref{E12}, we obtain
\begin{align}\label{sigma}
|\tau_{\mathbf s}|k^{\mathbf s}_\nu(z)+2\sigma_{\mathbf s} \frac{K^{\mathbf s}_\nu (z)}{P(z)}=0 \qquad \text{and } \qquad \sigma_{\mathbf s} k^{\mathbf s}_\nu(z)+4\frac{K^{\mathbf s}_\nu (z)}{P(z)} =0
\end{align}
which yield $\sigma_{\mathbf s}=-\sqrt{2|\tau_{\mathbf s}|}$ and \eqref{ode}.
\end{proof}
\hip
Since the sign of $\ell^{\mathbf s}$ was fixed by Proposition \ref{thetalisse} and $|\ell^{\mathbf s}_v|^2=1$, the choice of $\ell^{\mathbf s}$ is entirely determined by \eqref{choixl}.
Unfortunately, there is no CDF on $[\gamma_1+\nu,1]$ satisfying \eqref{ode}.
However, there exists a CDF on the whole segment $(\gamma_1,1]$ solving \eqref{ode}, which up to renormalization is given by
\begin{align}\label{k0}
K^{\mathbf s}_0(z)=\Big(\frac{z-\gamma_1}{z-\gamma_2}\Big)^{\frac{1}{2\sqrt{|\tau_{\mathbf s}|}}} \qquad \text{i.e } \qquad k^{\mathbf s}_0(z)=\frac{\gamma_1-\gamma_2}{2\sqrt{|\tau_{\mathbf s}|}(z-\gamma_2)^2}\Big(\frac{z-\gamma_1}{z-\gamma_2}\Big)^{\frac{1}{2\sqrt{|\tau_{\mathbf s}|}}-1}.
\end{align}
This leads to the introduction of the following CDF on $[\gamma_1+\nu,1]$ which will be an approximate solution of \eqref{ode}:
\begin{align}\label{knu}
K^{\mathbf s}_\nu(z)=\frac{ K_0^{\mathbf s}(z)-K_0^{\mathbf s}(\gamma_1+\nu)}{B_\nu^{\mathbf s}} \quad\qquad \text{and } \quad\qquad k^{\mathbf s}_\nu(z)=\big(K^{\mathbf s}_\nu\big)'(z)=\frac{k^{\mathbf s}_0(z)}{B_\nu^{\mathbf s}}
\end{align}
where
\begin{align}\label{bnu}
B_\nu^{\mathbf s}
= K_0^{\mathbf s}(1)-K_0^{\mathbf s}(\gamma_1+\nu)=K_0^{\mathbf s}(1)+O\Big(\nu\,{}^{\frac{1}{2\sqrt{|\tau_{\mathbf s}|}}}\Big).
\end{align}

\begin{lem}\label{petitom}
Recall the notation \eqref{tau} and let $(\ell^{\mathbf s})_{\mathbf s\in \mathbf j(\mathbf m)}$ a family of vectors satisfying \eqref{lv2}, \eqref{choixl} and whose signs are fixed by Proposition \ref{thetalisse}.
Let also $(k_\nu^{\mathbf s})_{\mathbf s\in \mathbf j(\mathbf m)}$ the probability densities on $[\gamma_1+\nu, 1]$ defined in \eqref{knu}.
Then for all $\mathbf s\in \mathbf j(\mathbf m)$ and $(x,v)\in B_0(\mathbf s,r)$, the prefactor from Proposition \ref{phf} satisfies
$$\om^{\mathbf m}_{\nu,z}(x,v)=O\Big(\nu\,{}^{\frac{1}{2\sqrt{|\tau_{\mathbf s}|}}}\Big) \begin{pmatrix}\partial_z L^{\mathbf s}_{z;x}\\\partial_z L^{\mathbf s}_{z;v}\end{pmatrix} \cdot \begin{pmatrix}\tilde x_{\mathbf s}\\v\end{pmatrix}.
$$
\end{lem}

\begin{proof}
By some computations similar to the ones we made in the proof of Lemma \ref{sisol}, we get that the choice of $(\ell^{\mathbf s})_{\mathbf s\in \mathbf j(\mathbf m)}$ implies that
$$\om^{\mathbf m}_{\nu,z}(x,v)=\frac{\sqrt{|\tau_{\mathbf s}|}P(z)}{2 \sqrt2} \bigg[k_\nu^{\mathbf s}(z)-\frac{2\sqrt 2}{\sqrt{|\tau_{\mathbf s}|}P(z)}K_\nu(z) \bigg]  \begin{pmatrix}\partial_z L^{\mathbf s}_{z;x}\\\partial_z L^{\mathbf s}_{z;v}\end{pmatrix} \cdot \begin{pmatrix}\tilde x_{\mathbf s}\\v\end{pmatrix}.$$
The term between brackets is exactly the one appearing in \eqref{ode} so using \eqref{knu} and the fact that $K_0$ is a solution of \eqref{ode}, we get 
$$\om^{\mathbf m}_{\nu,z}(x,v)=\frac{K_0^{\mathbf s}(\gamma_1+\nu)}{B_\nu^{\mathbf s}}  \begin{pmatrix}\partial_z L^{\mathbf s}_{z;x}\\\partial_z L^{\mathbf s}_{z;v}\end{pmatrix} \cdot \begin{pmatrix}\tilde x_{\mathbf s}\\v\end{pmatrix}=O\Big(\nu\,{}^{\frac{1}{2\sqrt{|\tau_{\mathbf s}|}}}\Big) \begin{pmatrix}\partial_z L^{\mathbf s}_{z;x}\\\partial_z L^{\mathbf s}_{z;v}\end{pmatrix} \cdot \begin{pmatrix}\tilde x_{\mathbf s}\\v\end{pmatrix}$$
by \eqref{bnu} and the definition of $K_0^{\mathbf s}$.
\end{proof}

\begin{prop}\label{pf}
Recall the notation \eqref{tau} and let $f_{\nu,h}^{\mathbf m}$ be the quasimode defined in \eqref{quasim} with $(\ell^{\mathbf s})_{\mathbf s\in \mathbf j(\mathbf m)}$ and $(k_\nu^{\mathbf s})_{\mathbf s\in \mathbf j(\mathbf m)}$ satisfying the hypotheses from Lemma \ref{petitom}.
Then
$$\|P_hf_{\nu,h}^{\mathbf m}\|=h\,\e^{-\frac{S(\mathbf m)}{h}}\|f^{\mathbf m}_{\nu,h}\|   \bigg( O_{\nu}\Big(h^{\frac{1}{2}}\Big)+O\Big(\nu\,{}^{\frac{1}{2\sqrt{|\tau_{\mathbf s}|}}}\,|\ln (\nu)|\Big)\bigg).$$
\end{prop}

\begin{proof}
First notice that thanks to item \ref{mseul} from Hypothesis \ref{jvide}, one can apply a standard Laplace method (see Proposition \ref{lap}) to obtain with the notation \eqref{hessiennes}
\begin{align}\label{f^2}
\|f^{\mathbf m}_{\nu,h}\|^2=\frac{(2\pi h)^d}{\det(\mathcal V_{\mathbf m})^{1/2}} \Big(1+O(h)\Big).
\end{align}
Hence, according to Proposition \ref{phf}, it is sufficient to show that
\begin{align}\label{amq}
\big\|P_hf_{\nu,h}^{\mathbf m}-R^{\mathbf m}_{\nu, h}\big\|=h^{1+\frac{d}{2}}\,\e^{-\frac{S(\mathbf m)}{h}}  O\Big(\nu\,{}^{\frac{1}{2\sqrt{|\tau_{\mathbf s}|}}}\,|\ln (\nu)|\Big).
\end{align}
Now, still using Proposition \ref{phf} as well as Minkowski's integral inequality and Lemma \ref{petitom}, we have
\begin{align}
\big\|P_hf_{\nu,h}^{\mathbf m}-R^{\mathbf m}_{\nu, h}\big\|     &\leq Ch^{1/2}   \int_{\gamma_1+\nu}^1 \bigg( \int_{\mathbf j^W(\mathbf m)+B_0(0,r)}\om^{\mathbf m}_{\nu,z}(x,v)^2\,\exp \Big[-\frac 2h \widetilde W^{\mathbf m}_z(x,v)\Big]  \D (x,v) \bigg)^{1/2}  \D z\\
		&\leq Ch^{1/2}\nu\,{}^{\frac{1}{2\sqrt{|\tau_{\mathbf s}|}}}\int_{\gamma_1+\nu}^1 \bigg( \sum_{\mathbf s \in \mathbf j(\mathbf m)} \int_{B_0(\mathbf s,r)}\begin{pmatrix}\partial_z L^{\mathbf s}_{z;x}\\\partial_z L^{\mathbf s}_{z;v}\end{pmatrix} \begin{pmatrix}\partial_z L^{\mathbf s}_{z;x}\\\partial_z L^{\mathbf s}_{z;v}\end{pmatrix}^t \begin{pmatrix}\tilde x_{\mathbf s}\\v\end{pmatrix} \cdot \begin{pmatrix}\tilde x_{\mathbf s}\\v\end{pmatrix}\\
		&\qquad \qquad \qquad \qquad\qquad \qquad\qquad \qquad\qquad \quad \; \times \exp \Big[-\frac 2h \widetilde W^{\mathbf m}_z(x,v)\Big]  \D (x,v) \bigg)^{1/2}  \D z.
\end{align}
With the notation \eqref{hessiennes}, the change of variables
$$(y,w)=\Big(\frac{2}{h}\Big)^{1/2} \bigg[\mathcal W_{\mathbf s}+ \begin{pmatrix} L^{\mathbf s}_{z;x}\\ L^{\mathbf s}_{z;v}\end{pmatrix} \begin{pmatrix} L^{\mathbf s}_{z;x}\\ L^{\mathbf s}_{z;v}\end{pmatrix}^t \bigg]^{1/2} (\tilde x_{\mathbf s},v)$$
then yields according to Lemma \ref{bonmin}
\begin{align}\label{yw}
\big\|P_hf_{\nu,h}^{\mathbf m}-R^{\mathbf m}_{\nu, h}\big\|     &\leq Ch^{1+\frac{d}{2}}\e^{-\frac{S(\mathbf m)}{h}}\nu\,{}^{\frac{1}{2\sqrt{|\tau_{\mathbf s}|}}}\\
						&\qquad \qquad \times\int_{\gamma_1+\nu}^1 \bigg( \sum_{\mathbf s \in \mathbf j(\mathbf m)} \int_{\R^{2d}}a_z a_z^t \begin{pmatrix}y\\w\end{pmatrix} \cdot \begin{pmatrix}y\\w\end{pmatrix} \e^{-\frac{(y,w)^2}{2}} \D (y,w) \bigg)^{1/2}  \D z
\end{align}
where
$$a_z= \bigg[\mathcal W_{\mathbf s}+ \begin{pmatrix} L^{\mathbf s}_{z;x}\\ L^{\mathbf s}_{z;v}\end{pmatrix} \begin{pmatrix} L^{\mathbf s}_{z;x}\\ L^{\mathbf s}_{z;v}\end{pmatrix}^t \bigg]^{-1/2}\begin{pmatrix}\partial_z L^{\mathbf s}_{z;x}\\\partial_z L^{\mathbf s}_{z;v}\end{pmatrix}.$$
Thanks to Proposition \ref{mom2}, we know that
\begin{align}\label{aa}
(2\pi )^{-d}\int_{\R^{2d}}a_z a_z^t \begin{pmatrix}y\\w\end{pmatrix} \cdot \begin{pmatrix}y\\w\end{pmatrix} \e^{-\frac{(y,w)^2}{2}} \D (y,w)=|a_z|^2=\bigg[\mathcal W_{\mathbf s}+ \begin{pmatrix} L^{\mathbf s}_{z;x}\\ L^{\mathbf s}_{z;v}\end{pmatrix} \begin{pmatrix} L^{\mathbf s}_{z;x}\\ L^{\mathbf s}_{z;v}\end{pmatrix}^t \bigg]^{-1}\begin{pmatrix}\partial_z L^{\mathbf s}_{z;x}\\\partial_z L^{\mathbf s}_{z;v}\end{pmatrix} \cdot \begin{pmatrix}\partial_z L^{\mathbf s}_{z;x}\\\partial_z L^{\mathbf s}_{z;v}\end{pmatrix}.
\end{align}
Since by \eqref{dyl}
$$\bigg[\mathcal W_{\mathbf s}+ \begin{pmatrix} L^{\mathbf s}_{z;x}\\ L^{\mathbf s}_{z;v}\end{pmatrix} \begin{pmatrix} L^{\mathbf s}_{z;x}\\ L^{\mathbf s}_{z;v}\end{pmatrix}^t \bigg]^{-1}\begin{pmatrix}\partial_z L^{\mathbf s}_{z;x}\\\partial_z L^{\mathbf s}_{z;v}\end{pmatrix}=\frac{-8}{P(z)^{3/2}}\begin{pmatrix}\big(2|\tau_{\mathbf s}|\big)^{-1/2}(1-z)\,\ell_v^{\mathbf s}\\(1+z)\,\ell_v^{\mathbf s}\end{pmatrix},$$
we get
\begin{align}\label{aa2}
|a_z|^2=\frac{16}{P(z)^3}\Big(2(1+z)^2-(1-z)^2\Big)=\frac{16}{P(z)^2}.
\end{align}
Putting together \eqref{yw}, \eqref{aa}, \eqref{aa2} and computing the integral in $z$, we obtain \eqref{amq} so the proof is complete.
\end{proof}

\section{Computation of the approximated small eigenvalues}\label{sectionvp}
\hip
Let us denote 
\begin{align}\label{ftilde}
\tilde f^{\mathbf m}_{\nu,h}=\frac{f^{\mathbf m}_{\nu,h}}{\|f^{\mathbf m}_{\nu,h}\|}
\end{align}
the renormalization of the quasimodes defined in \eqref{quasim} and satisfying the hypotheses of Proposition \ref{pf}.
The goal of this section is to compute the \emph{approximated eigenvalues}
\begin{align}\label{lamtilde}
\tilde \lambda^{\mathbf m}_{\nu,h}:=\langle P_h \tilde f^{\mathbf m}_{\nu,h},\tilde f^{\mathbf m}_{\nu,h} \rangle=\langle Q_h\tilde f^{\mathbf m}_{\nu,h},\tilde f^{\mathbf m}_{\nu,h} \rangle
\end{align}
as $X_0^h$ is a skew-adjoint differential operator and $\tilde f^{\mathbf m}_{\nu,h}$ is real valued.\\
This will require to study the matrix
\begin{align}\label{HY}
H^{\mathbf s}_\gamma=\begin{pmatrix}  \mathcal V_{\mathbf s}+2L_{\gamma,x} L_{\gamma,x}^t & L_{\gamma,x}L_{\gamma,v}^t & L_{\gamma,x}L_{\gamma,v}^t \\
L_{\gamma,v} L_{\gamma,x}^t & \frac12+L_{\gamma,v} L_{\gamma,v}^t & 0 \\
L_{\gamma,v}L_{\gamma,x}^t &  0 &  \frac12+L_{\gamma,v}L_{\gamma,v}^t \end{pmatrix}
\end{align}
where we used the notation \eqref{hessiennes} and for shortness, we wrote $L_{\gamma,x}$ and $L_{\gamma,v}$ instead of $L_{\gamma,x}^{\mathbf s}$ and $L_{\gamma,v}^{\mathbf s}$.

\begin{lem}\label{hy}
For $\gamma\in [\gamma_1+\nu,1]$, the matrix $H_\gamma^{\mathbf s}$ 
is positive definite.
\end{lem}

\begin{proof}
It suffices to notice that 
$$H_\gamma^{\mathbf s} \begin{pmatrix}
x\\v\\v'
\end{pmatrix}\cdot\begin{pmatrix}
x\\v\\v'
\end{pmatrix}=
\bigg[\mathcal W_{\mathbf s}+\begin{pmatrix}L^{\mathbf s}_{\gamma;x}\\L^{\mathbf s}_{\gamma;v}\end{pmatrix}\begin{pmatrix}L^{\mathbf s}_{\gamma;x}\\L^{\mathbf s}_{\gamma;v}\end{pmatrix}^t\bigg]\begin{pmatrix}
x\\v
\end{pmatrix}\cdot\begin{pmatrix}
x\\v
\end{pmatrix}+
\bigg[\mathcal W_{\mathbf s}+\begin{pmatrix}L^{\mathbf s}_{\gamma;x}\\L^{\mathbf s}_{\gamma;v}\end{pmatrix}\begin{pmatrix}L^{\mathbf s}_{\gamma;x}\\L^{\mathbf s}_{\gamma;v}\end{pmatrix}^t\bigg]\begin{pmatrix}
x\\v'
\end{pmatrix}\cdot\begin{pmatrix}
x\\v'
\end{pmatrix}
$$
and apply Lemma \ref{bonmin}.
\end{proof}
\hip
In the spirit of Proposition \ref{facto} and with the notation \eqref{my}, let us denote
\begin{align}\label{qy}
Q_{y,h}=b_h^*\circ \mathrm{Op}_h(m_{y,h} \, \mathrm{Id})\circ b_h. 
\end{align}
For $\mathbf m\in \mathtt U^{(0)}\backslash \{\underline{\mathbf m}\}$, $\mathbf s\in \mathbf j(\mathbf m)$, 
we also denote 
$\langle \cdot,\cdot\rangle_{\tilde r}$ the inner product on $L^2\big(B(\mathbf s,\tilde r)\times \R^d_v\big)$.

\begin{lem}\label{qygammagamma}
Let $\mathbf s\in \mathbf j(\mathbf m)$ for some $\mathbf m\in \mathtt U^{(0)}\backslash \{\underline{\mathbf m}\}$ and recall the notations \eqref{hessiennes}, \eqref{wm} and \eqref{thetatilde}. 
Then for all $\gamma \in [\gamma_1+\nu,1]$ and $y \in (0,1)$,
$$\Big\langle Q_{y,h}\Big(\tilde \theta^{\mathbf s}_{\gamma,h}\e^{-\frac{W^{\mathbf m}(x,v)}{h}}\Big),\tilde \theta^{\mathbf s}_{\gamma,h}\e^{-\frac{W^{\mathbf m}(x,v)}{h}}\Big\rangle_{\tilde r}=2h^2\e^{\frac{-2S(\mathbf m)}{h}}  (2\pi h)^{d}|\det \mathcal V_{\mathbf s}|^{-1/2}  \frac{1+O_\nu (h)}{(1+y)\Big(1+\big(1+2|L_{\gamma,v}^{\mathbf s}|^2\big)y\Big)}  |L_{\gamma,v}^{\mathbf s}|^2.$$
\end{lem}

\begin{proof}
First, let us use the definition of $Q_{y,h}$ to write
$$\Big\langle Q_{y,h}\Big(\tilde \theta^{\mathbf s}_{\gamma,h}\e^{-\frac{W^{\mathbf m}(x,v)}{h}}\Big),\tilde \theta^{\mathbf s}_{\gamma,h}\e^{-\frac{W^{\mathbf m}(x,v)}{h}}\Big\rangle_{\tilde r}=\Big\langle \mathrm{Op}_h(m_{y,h} \, \mathrm{Id})\circ b_h\Big(\tilde \theta^{\mathbf s}_{\gamma,h}\e^{-\frac{W^{\mathbf m}(x,v)}{h}}\Big),b_h\Big(\tilde \theta^{\mathbf s}_{\gamma,h}\e^{-\frac{W^{\mathbf m}(x,v)}{h}}\Big)\Big\rangle_{\tilde r}.$$
Using \eqref{opm2}, we get
\begin{align}\label{Qythetagamma}
\Big\langle Q_{y,h}&\Big(\tilde \theta^{\mathbf s}_{\gamma,h}\e^{-\frac{W^{\mathbf m}(x,v)}{h}}\Big),\tilde \theta^{\mathbf s}_{\gamma,h}\e^{-\frac{W^{\mathbf m}(x,v)}{h}}\Big\rangle_{\tilde r}=2h^2(2\pi h)^{-d/2} \frac{(y+1)^{d-2}}{(4y)^{\frac d2}}\e^{\frac{V(\mathbf m)}{h}} |L_{\gamma,v}^{\mathbf s}|^2\times  \\
		\int_{|x-\mathbf s|<\tilde r,\,v, v'\in \R^d}&\exp\Big[-\frac 1h\Big(V(x)+\frac{v^2+v'^2}{4}+\frac y8 (v+v')^2+\frac{(v-v')^2}{8y}+\frac{L^{\mathbf s}_\gamma(x,v)^2+L^{\mathbf s}_\gamma(x,v')^2}{2}\Big)\Big]  \D x \D v \D v'  . \nonumber
\end{align}
By the change of variables $x'=\phi_{\mathbf s}(x)$ and with the notation $\boldsymbol{\sigma}(\mathbf m)$ from Definition \ref{j et s}, the last integral becomes
\begin{align}\label{hyxx}
\e^{\frac{-2\boldsymbol{\sigma}(\mathbf m)}{h}}\int_{|\phi^{-1}_{\mathbf s}(x')-\mathbf s|<\tilde r,\,v,\, v'\in \R^d}\exp\Big[-\frac {1}{2h} H_{\gamma,y}^{\mathbf s} \begin{pmatrix}x'\\v\\v' \end{pmatrix}\cdot \begin{pmatrix}x'\\v\\v' \end{pmatrix}\Big] |\det D_{x'}\phi^{-1}_{\mathbf s}|  \,\D x' \D v \D v'
\end{align}
where using the notation \eqref{HY},
\begin{align}\label{hy>}
H_{\gamma,y}^{\mathbf s}=\begin{pmatrix}  \mathcal V_{\mathbf s}+2L_{\gamma,x} L_{\gamma,x}^t & L_{\gamma,x}L_{\gamma,v}^t & L_{\gamma,x}L_{\gamma,v}^t \\
L_{\gamma,v} L_{\gamma,x}^t & \frac{(y+1)^2}{4y} +L_{\gamma,v} L_{\gamma,v}^t & \frac{y^2-1}{4y} \\
L_{\gamma,v}L_{\gamma,x}^t &  \frac{y^2-1}{4y} & \frac{(y+1)^2}{4y}+ L_{\gamma,v}L_{\gamma,v}^t \end{pmatrix}=H_\gamma^{\mathbf s}+ \begin{pmatrix}
0&0&0\\
0&\frac{y^2+1}{4y}&\frac{y^2-1}{4y}\\
0&\frac{y^2-1}{4y}&\frac{y^2+1}{4y}
\end{pmatrix}
\end{align}
is a positive-definite matrix uniformly in $(\gamma,y)\in [\gamma_1+\nu,1]\times(0,1)$ thanks to Lemma \ref{hy}.
Hence, $(H_{\gamma,y}^{\mathbf s})^{-1/2}$ exists and is $O_\nu(1)$ so by a
standard Laplace method (see Proposition \ref{lap}),
\begin{align}\label{hyxx2}
\int_{|\phi^{-1}_{\mathbf s}(x')-\mathbf s|<\tilde r,\,v,\, v'\in \R^d}\exp\Big[-\frac {1}{2h} H_{\gamma,y}^{\mathbf s} \begin{pmatrix}x'\\v\\v' \end{pmatrix}\cdot \begin{pmatrix}x'\\v\\v' \end{pmatrix}\Big] |\det D_{x'}\phi^{-1}_{\mathbf s}|  \,\D x' \D v \D v'=(2\pi h)^{3d/2}\det (H_{\gamma,y}^{\mathbf s})^{-1/2}& \nonumber\\
	 \times \big(1+O_\nu (&h)\big)\nonumber\\
	=(2\pi h)^{3d/2}|\det \mathcal V_{\mathbf s}|^{-1/2}\frac{(4y)^{d/2}}{(1+y)^{d-1}\Big(1+\big(1+2|L_{\gamma,v}^{\mathbf s}|^2\big)y\Big)}\big(1+O_\nu (&h)\big)
\end{align}
where we also used Lemma \ref{det2}. The conclusion then follows from \eqref{Qythetagamma}, \eqref{hyxx} and \eqref{hyxx2}.
\end{proof}

\begin{lem}\label{qythetagamma}
Recall the notation \eqref{Gamma} and let $\gamma_1+\nu\leq z\leq \gamma < 1$.
For $y\in [0,1)$, we have
$$\Gamma_z^{-1} \circ \Gamma_\gamma(y)\in [0,1)$$
and
$$Q_{y,h}\Big(\tilde \theta^{\mathbf s}_{\gamma,h}\e^{-\frac{W^{\mathbf m}(x,v)}{h}}\Big)=(\Gamma_z^{-1} \circ \Gamma_\gamma)'(y)\,Q_{\Gamma_z^{-1} \circ \Gamma_\gamma(y),h}\Big(\tilde \theta^{\mathbf s}_{z,h}\e^{-\frac{W^{\mathbf m}(x,v)}{h}}\Big)$$
on $B(\mathbf s,\tilde r)\times \R^d_v$.
\end{lem}

\begin{proof}
First, notice that for all $\gamma\in [\gamma_1+\nu,1)$, the function $\Gamma_\gamma:[0,1)\to [\gamma,1)$ is an increasing bijection whose inverse is given by
\begin{align}\label{Gamma-1}
\Gamma_\gamma^{-1}(y)=\frac{y-\gamma}{1-y\gamma}
\end{align}
so the first assertion follows from the hypothesis on $z$ and $\gamma$.
Now, by Lemma \ref{actionq2} applied with $Q_{y,h}$ instead of $Q_h$, 
we get using the notation \eqref{Ltilde} as well as \eqref{Lcircgamma} that on $B(\mathbf s,\tilde r)\times \R^d_v$,
\begin{align}\label{Qythetaz}
Q_{y,h}\Big(\tilde \theta^{\mathbf s}_{\gamma,h}\e^{-\frac{W^{\mathbf m}(x,v)}{h}}\Big)=-h\, \partial_y \mathscr L(\gamma,y)\,\e^{-\frac{\widetilde W_{\Gamma_\gamma(y)}(x,v)}{h}}\,\cdot \begin{pmatrix}\tilde x_{\mathbf s}\\v\end{pmatrix}
\end{align}
(here we once again disregarded the fact that the linear form $L_\gamma$ is twisted in $x$ as $Q_{y,h}$ only acts in $v$).
Thus, denoting $\partial_2 \mathscr L(\gamma,\cdot)$ the derivative of $\mathscr L$ w.r.t its second argument and still using \eqref{Lcircgamma}, we also have 
\begin{align}
Q_{y,h}\Big(\tilde \theta^{\mathbf s}_{\gamma,h}\e^{-\frac{W^{\mathbf m}(x,v)}{h}}\Big)&=-h\, \partial_y(L_{\Gamma_{\gamma}(y)})\,\e^{-\frac{\widetilde W_{\Gamma_\gamma(y)}(x,v)}{h}}\,\cdot \begin{pmatrix}\tilde x_{\mathbf s}\\v\end{pmatrix}\\
		&=-h\, \partial_y\Big(\mathscr L\big(z,\Gamma_z^{-1}\circ\Gamma_{\gamma}(y)\big)\Big)\,\e^{-\frac{\widetilde W_{\Gamma_\gamma(y)}(x,v)}{h}}\,\cdot \begin{pmatrix}\tilde x_{\mathbf s}\\v\end{pmatrix}\\
				&=-h \,(\Gamma_z^{-1}\circ\Gamma_{\gamma})'(y)\, \partial_2\mathscr L\big(z,\Gamma_z^{-1}\circ\Gamma_{\gamma}(y)\big)\,\e^{-\frac{\widetilde W_{\Gamma_\gamma(y)}(x,v)}{h}}\,\cdot \begin{pmatrix}\tilde x_{\mathbf s}\\v\end{pmatrix}
\end{align}
so \eqref{Qythetaz} with $Q_{\Gamma_z^{-1} \circ \Gamma_\gamma(y),h}$ and $\tilde \theta^{\mathbf s}_{z,h}$ yields the last statement.
\end{proof}

\begin{prop}\label{Pff}
With the notations \eqref{hessiennes}, \eqref{tau}, \eqref{k0} and \eqref{lamtilde}, we have for $\mathbf m\in \mathtt U^{(0)}\backslash \{\underline{\mathbf m}\}$
$$\tilde \lambda^{\mathbf m}_{\nu,h}=h\, \tilde \varrho^{}_{\nu,h}(\mathbf m)\, \e^{\frac{-2S(\mathbf m)}{h}}$$
with
\begin{align}
\tilde \varrho_{\nu,h}(\mathbf m)=&\frac{1}{\pi}\sum_{\mathbf s\in \mathbf j(\mathbf m)}\bigg(\frac{2+\sqrt2}{2-\sqrt2}\bigg)^{\frac{1}{\sqrt{|\tau_{\mathbf s}|}}}\bigg(\frac{\det \mathcal V_{\mathbf m}}{|\det \mathcal V_{\mathbf s}|}\bigg)^{1/2}   \int_{\gamma_1\leq z\leq \gamma < 1} k_0^{\mathbf s}(\gamma) k_0^{\mathbf s}(z) \ln \bigg(2\,\frac{(1+z)(1+\gamma)}{1+3z+3\gamma+z\gamma}\bigg) \, \D z \D \gamma\\
		& \qquad+O_\nu\big(h\big)+O\Big(\nu^{\frac{1}{2\sqrt{|\tau_{\mathbf s}|}}}\Big).
\end{align}
\end{prop}

\begin{proof}
As we mentionned at the begining of the section, since $X_0^h$ is a skew-adjoint differential operator and $\tilde f^{\mathbf m}_{\nu,h}$ is real valued, we have
$$\langle P_h \tilde f^{\mathbf m}_{\nu,h},\tilde f^{\mathbf m}_{\nu,h} \rangle=\langle Q_h\tilde f^{\mathbf m}_{\nu,h},\tilde f^{\mathbf m}_{\nu,h} \rangle.$$
Now by Proposition \ref{facto}, we get
\begin{align}\label{qff}
\langle Q_h f^{\mathbf m}_{\nu,h}, f^{\mathbf m}_{\nu,h} \rangle= \big\langle \mathrm{Op}_h (m_{h}\,\mathrm{Id}) \big( b_h f^{\mathbf m}_{\nu,h}\big),b_h f^{\mathbf m}_{\nu,h} \big\rangle
\end{align}
and we saw through \eqref{qf}-\eqref{bhfsimple} that
\begin{align}
b_hf^{\mathbf m}_{\nu,h}&=\Big(\frac {h}{2\pi}\Big)^{1/2}  \sum_{\mathbf s\in \mathbf j(\mathbf m)}  \int_{\gamma_1+\nu}^{1} k_\nu^{\mathbf s}(\gamma)   \e^{-\frac{\widetilde W^{\mathbf m}_{\gamma}}{h}}  L_{\gamma,v}^{\mathbf s}  \D \gamma\, \1_{|x-\mathbf s|<\tilde r}+O_\nu\Big(h^{\frac{3+d}{2}}\e^{-\frac{S(\mathbf m)}{h}}\Big)\label{bh1}\\
		&=(2\pi h)^{-1/2}  \sum_{\mathbf s\in \mathbf j(\mathbf m)}  \int_{\gamma_1+\nu}^{1} k_\nu^{\mathbf s}(\gamma) b_h\Big(\tilde\theta_{\gamma,h}\e^{-\frac{W^{\mathbf m}(x,v)}{h}}\Big)   \D \gamma\, \1_{|x-\mathbf s|<\tilde r}+O_\nu\Big(h^{\frac{3+d}{2}}\e^{-\frac{S(\mathbf m)}{h}}\Big) \label{bhf=bhtheta}
\end{align}
Note that \eqref{bh1} also implies
\begin{align}
b_hf^{\mathbf m}_{\nu,h}=O_\nu\Big(h^{\frac{1+d}{2}}\e^{-\frac{S(\mathbf m)}{h}}\Big). \label{bhf=o}
\end{align}
Combining the boundedness of $\mathrm{Op}_h (m_{h}\,\mathrm{Id})$ with \eqref{bhf=bhtheta}-\eqref{bhf=o} and using the notation \eqref{qy}, \eqref{qff} becomes
\begin{align}
\langle Q_h f^{\mathbf m}_{\nu,h}, f^{\mathbf m}_{\nu,h} &\rangle=(2\pi h)^{-1}  \sum_{\mathbf s\in \mathbf j(\mathbf m)} \int_{[\gamma_1+\nu,1]^2}  k_\nu^{\mathbf s}(\gamma) k_\nu^{\mathbf s}(z)\Big\langle Q_h \Big(\tilde\theta_{\gamma,h}\e^{-\frac{W^{\mathbf m}(x,v)}{h}}\Big) ,\tilde\theta_{z,h}\e^{-\frac{W^{\mathbf m}(x,v)}{h}} \Big\rangle_{\tilde r} \,\D \gamma \D z  \\
		&\qquad+O_\nu\Big(h^{d+2}\e^{-\frac{2S(\mathbf m)}{h}}\Big)\\
		&=(2\pi h)^{-1}  \sum_{\mathbf s\in \mathbf j(\mathbf m)}\int_0^1 \int_{[\gamma_1+\nu,1]^2}  k_\nu^{\mathbf s}(\gamma) k_\nu^{\mathbf s}(z)\Big\langle Q_{y,h} \Big(\tilde\theta_{\gamma,h}\e^{-\frac{W^{\mathbf m}(x,v)}{h}}\Big) ,\tilde\theta_{z,h}\e^{-\frac{W^{\mathbf m}(x,v)}{h}} \Big\rangle_{\tilde r} \,\D \gamma \D z \D y \\
		&\qquad+O_\nu\Big(h^{d+2}\e^{-\frac{2S(\mathbf m)}{h}}\Big)\\
		&=2(2\pi h)^{-1}  \sum_{\mathbf s\in \mathbf j(\mathbf m)}\int_0^1 \int_{\gamma_1+\nu\leq z\leq \gamma < 1}  k_\nu^{\mathbf s}(\gamma) k_\nu^{\mathbf s}(z)\Big\langle Q_{y,h} \Big(\tilde\theta_{\gamma,h}\e^{-\frac{W^{\mathbf m}(x,v)}{h}}\Big) ,\tilde\theta_{z,h}\e^{-\frac{W^{\mathbf m}(x,v)}{h}} \Big\rangle_{\tilde r} \,\D z \D \gamma \D y \label{qygammaz}\\
		&\qquad+O_\nu\Big(h^{d+2}\e^{-\frac{2S(\mathbf m)}{h}}\Big)
\end{align}
where for the last equation we used the fact that $Q_{y,h}$ is self-adjoint.
Applying Lemma \ref{qythetagamma} together with the change of variables $\tilde y=\Gamma_z^{-1} \circ \Gamma_\gamma(y)$, we get that \eqref{qygammaz} yields
\begin{align}
\langle Q_h &f^{\mathbf m}_{\nu,h}, f^{\mathbf m}_{\nu,h} \rangle+O_\nu\Big(h^{d+2}\e^{-\frac{2S(\mathbf m)}{h}}\Big)=\\
	&2(2\pi h)^{-1}  \sum_{\mathbf s\in \mathbf j(\mathbf m)} \int_{\gamma_1+\nu\leq z\leq \gamma < 1} \int_{\Gamma_z^{-1}(\gamma)}^1  k_\nu^{\mathbf s}(\gamma) k_\nu^{\mathbf s}(z)\Big\langle Q_{\tilde y,h} \Big(\tilde\theta_{z,h}\e^{-\frac{W^{\mathbf m}(x,v)}{h}}\Big) ,\tilde\theta_{z,h}\e^{-\frac{W^{\mathbf m}(x,v)}{h}} \Big\rangle_{\tilde r} \,\D \tilde y \D z \D \gamma
\end{align}
which by Lemma \ref{qygammagamma} is further equal to
\begin{align}\label{3inte}
\text{}\qquad\frac{2}{\pi} h (2\pi h)^{d} \e^{\frac{-2S(\mathbf m)}{h}}  \sum_{\mathbf s\in \mathbf j(\mathbf m)} |\det \mathcal V_{\mathbf s}|^{-1/2} \int_{\gamma_1+\nu\leq z\leq \gamma < 1} \int_{\Gamma_z^{-1}(\gamma)}^1  \frac{ k_\nu^{\mathbf s}(\gamma) k_\nu^{\mathbf s}(z) |L_{z,v}^{\mathbf s}|^2}{(1+\tilde y)\Big(1+\big(1+2|L_{z,v}^{\mathbf s}|^2\big)\tilde y\Big)}  
 \,\D \tilde y \D z \D \gamma.
\end{align}
By partial fraction decomposition, the $\tilde y$-integral becomes
\begin{align}
\int_{\Gamma_z^{-1}(\gamma)}^1  \frac{1}{(1+\tilde y)\Big(1+\big(1+2|L_{z,v}^{\mathbf s}|^2\big)\tilde y\Big)} \,\D \tilde y&=\frac{1}{2|L_{z,v}^{\mathbf s}|^2}\int_{\Gamma_z^{-1}(\gamma)}^1  \frac{1+2|L_{z,v}^{\mathbf s}|^2}{1+\big(1+2|L_{z,v}^{\mathbf s}|^2\big)\tilde y}- \frac{1}{1+\tilde y} \,\D \tilde y\\
		&=\frac{1}{2|L_{z,v}^{\mathbf s}|^2} \ln \Bigg( \frac{\big(1+|L_{z,v}^{\mathbf s}|^2\big)\big(1+\Gamma_z^{-1}(\gamma)\big)}{1+\big(1+2|L_{z,v}^{\mathbf s}|^2\big)\Gamma_z^{-1}(\gamma)} \Bigg) \label{logmoche}
\end{align}
and using \eqref{P}-\eqref{defL} as well as \eqref{Gamma-1}, the quantity in the logarithm from \eqref{logmoche} simplifies as follows:
\begin{align}
\frac{\big(1+|L_{z,v}^{\mathbf s}|^2\big)\big(1+\Gamma_z^{-1}(\gamma)\big)}{1+\big(1+2|L_{z,v}^{\mathbf s}|^2\big)\Gamma_z^{-1}(\gamma)}&=\frac{\big(P(z)+(1-z)^2\big)(1-z)(1+\gamma)}{P(z)(1-\gamma z) +(3z^2+2z+3)(\gamma-z)}\\
		&=2\,\frac{(1+z)^2(1-z)(1+\gamma)}{(1-z^2)(1+3z+3\gamma+z\gamma)}\\
		&=2\,\frac{(1+z)(1+\gamma)}{1+3z+3\gamma+z\gamma}. \label{logbeau}
\end{align}
Putting together \eqref{3inte}, \eqref{logmoche}, \eqref{logbeau} and using \eqref{f^2}, we get
\begin{align}
&\langle P_h \tilde f^{\mathbf m}_{\nu,h},\tilde f^{\mathbf m}_{\nu,h} \rangle+O_\nu\Big(h^{2}\e^{-\frac{2S(\mathbf m)}{h}}\Big)=\\
		&\qquad\qquad\frac{h}{\pi}\, \e^{\frac{-2S(\mathbf m)}{h}}    \sum_{\mathbf s\in \mathbf j(\mathbf m)} \bigg(\frac{\det \mathcal V_{\mathbf m}}{|\det \mathcal V_{\mathbf s}|}\bigg)^{1/2} \int_{\gamma_1+\nu\leq z\leq \gamma < 1} k_\nu^{\mathbf s}(\gamma) k_\nu^{\mathbf s}(z) \ln \bigg(2\,\frac{(1+z)(1+\gamma)}{1+3z+3\gamma+z\gamma}\bigg) \, \D z \D \gamma. \label{intknu}
\end{align}
Now, the function $1+3z+3\gamma+z\gamma$ is non-negative on $[\gamma_1,1]^2$ and vanishes only at $(\gamma_1,\gamma_1)$.
Moreover, we have by Taylor expansion that
$$1+3z+3\gamma+z\gamma\geq \frac{|(\gamma,z)-(\gamma_1,\gamma_1)|}{C}\geq \max \Big( \frac{z-\gamma_1}{C} , \frac{\gamma-\gamma_1}{C} \Big)$$
for $(\gamma,z)\in [\gamma_1,1]^2$ close enough to $(\gamma_1,\gamma_1)$ and thus
$$\ln \bigg(2\,\frac{(1+z)(1+\gamma)}{1+3z+3\gamma+z\gamma}\bigg)=O\big(|\ln (z-\gamma_1)|\big)$$
holds as well as
$$\ln \bigg(2\,\frac{(1+z)(1+\gamma)}{1+3z+3\gamma+z\gamma}\bigg)=O\big(|\ln (\gamma-\gamma_1)|\big).$$
Besides, by \eqref{knu} and \eqref{bnu}, we have 
\begin{align}\label{knuk0}
k^{\mathbf s}_\nu(z)=\bigg(\frac{2-\sqrt2}{2+\sqrt2}\bigg)^{\frac{-1}{2\sqrt{|\tau_{\mathbf s}|}}}k^{\mathbf s}_0(z)\bigg(1+O\Big(\nu ^{\frac{1}{2\sqrt{|\tau_{\mathbf s}|}}}\Big)\bigg)
\end{align}
with $k^{\mathbf s}_0(z)=O\Big(|z-\gamma_1|^{\frac{1}{2\sqrt{|\tau_{\mathbf s}|}}-1}\Big)$ on $[\gamma_1,1]$.
Consequently, the integral
$$\int_{\gamma_1\leq z\leq \gamma < 1} k_0^{\mathbf s}(\gamma) k_0^{\mathbf s}(z) \ln \bigg(2\,\frac{(1+z)(1+\gamma)}{1+3z+3\gamma+z\gamma}\bigg) \, \D z \D \gamma$$
exists and we have
\begin{align}
\int_{\gamma_1+\nu\leq z\leq \gamma < 1} k_0^{\mathbf s}(\gamma) k_0^{\mathbf s}(z)& \ln \bigg(2\,\frac{(1+z)(1+\gamma)}{1+3z+3\gamma+z\gamma}\bigg) \, \D z \D \gamma+O\Big(\nu ^{\frac{1}{2\sqrt{|\tau_{\mathbf s}|}}}\Big) \label{intk0} \\
		&=\int_{\gamma_1\leq z\leq \gamma < 1} k_0^{\mathbf s}(\gamma) k_0^{\mathbf s}(z) \ln \bigg(2\,\frac{(1+z)(1+\gamma)}{1+3z+3\gamma+z\gamma}\bigg) \, \D z \D \gamma. 
\end{align}
Combining \eqref{intknu}, \eqref{knuk0} and \eqref{intk0}, we get the announced result.
\end{proof}

\section{Proof of the main results}

\hip
We now introduce a series of results which will enable us to go from the approximated eigenvalues of $P_h$ to the actual ones.

\begin{lem}\label{Pf^2}
Let $\mathbf m \in \mathtt U^{(0)}\backslash \{\underline{\mathbf m}\}$.
Using the notations \eqref{tau}, \eqref{ftilde} and \eqref{lamtilde}, we have 
\begin{enumerate}[label=\roman*)]
\item $\|P_h\tilde f^{\mathbf m}_{\nu ,h}\|=\sqrt{h\tilde \lambda^{\mathbf m}_{\nu,h}}\, \bigg( O_{\nu}\Big(h^{\frac{1}{2}}\Big)+O\Big(\nu\,{}^{\frac{1}{2\sqrt{|\tau_{\mathbf s}|}}}\,|\ln (\nu)|\Big)\bigg)$
\item $\|P_h^*\tilde f^{\mathbf m}_{\nu ,h}\|=\sqrt{h\tilde \lambda^{\mathbf m}_{\nu,h}}\, \bigg( O_{\nu}\Big(h^{\frac{1}{2}}\Big)+O\Big(|\ln (\nu)|\Big)\bigg)$.
\end{enumerate}
\end{lem}

\begin{proof} 
The first item is an immediate consequence of Propositions \ref{pf} and \ref{Pff}.
The second one can be obtained similarly using Remark \ref{Ph*f} and mimicking the proof of Proposition \ref{pf} after noticing that
$$\overset{*}{\om}{}^{\mathbf m}_{\nu,z}(x,v)=O(1) \begin{pmatrix}\partial_z L^{\mathbf s}_{z;x}\\\partial_z L^{\mathbf s}_{z;v}\end{pmatrix} \cdot \begin{pmatrix}\tilde x_{\mathbf s}\\v\end{pmatrix}.$$
\end{proof}

\begin{lem}\label{f1f2}
For $\mathbf m$ and $\mathbf m'$ two distinct elements of $\mathtt U^{(0)}$, we have
\begin{enumerate}[label=\roman*)]
\item $\langle P_h\tilde f^{\mathbf m}_{\nu,h}, \tilde f^{\mathbf m'}_{\nu,h} \rangle=O_\nu\Big(h^\infty\sqrt{\tilde \lambda^{\mathbf m}_{\nu,h}\tilde \lambda^{\mathbf m'}_{\nu,h}}\,\Big)$
\item There exists $c>0$ such that $\big\langle \tilde f^{\mathbf m}_{\nu,h}, \tilde f^{\mathbf m'}_{\nu,h} \big\rangle=O(\e^{-c/h})$ \label{f1f22}
\end{enumerate}
\end{lem}

\begin{proof} 
The proof is a straightforward adaptation of the one of Lemma 5.5 in \cite{me}, even though the operator $P_h$ and the quasimodes $(\tilde f^{\mathbf m}_{\nu,h})_{\mathbf m}$ differ from the ones of this reference.
We recall the main steps for the reader's convenience.\\
i): The idea is to use \eqref{suppf}, the fact that $P_h$ is local in $x$, Hypothesis \ref{jvide} and the support properties of $\nabla \theta^{\mathbf m}_{\nu,h}$ and $\nabla \chi_{\mathbf m}$ to show that 
\begin{align}
\Big|\Big\langle P_h f^{\mathbf m}_{\nu,h},  f^{\mathbf m'}_{\nu,h} \Big \rangle\Big|\leq \Big\langle \mathrm{Op}_h(m_h\, \mathrm{Id})\big(\theta^{\mathbf m}_{\nu,h} (\partial_v \chi_{\mathbf m}) \e^{-W^{\mathbf m}/h}\big)\, , \, b_hf_{\mathbf m',h} \Big \rangle=O_\nu\Big(h^\infty \e^{-\frac{S(\mathbf m)+S(\mathbf m')}{h}}\Big)
\end{align}
by \eqref{bhf=o}. We can then conclude with \eqref{f^2}.
\hip 
ii): It is shown in \cite{me} (proof of Lemma 5.5) that when $V(\mathbf m)=V(\mathbf m')$, the supports of $f^{\mathbf m}_{\nu,h}$ and $f^{\mathbf m'}_{\nu,h}$ do not meet.
Thus we can suppose that $V(\mathbf m)>V(\mathbf m')$ and in that case, using once again \eqref{suppf} and Hypothesis \ref{jvide}, we show that
\begin{align*}
\Big\langle  f^{\mathbf m}_{\nu,h},  f^{\mathbf m'}_{\nu,h} \Big\rangle=\int_{E(\mathbf m)+B(0,\varepsilon')} \theta^{\mathbf m}_{\nu,h}\theta^{\mathbf m'}_{\nu,h}\chi_{\mathbf m}\chi_{\mathbf m'} \e^{-\frac{2V-V(\mathbf m)-V(\mathbf m')+v^2}{2h}}\D(x,v)=O\Big(\e^{-\frac{V(\mathbf m)-V(\mathbf m')}{2h}}\Big)
\end{align*}
so the conclusion immediately follows from \eqref{f^2}.
\end{proof}
\hip
In order to go from quasimodes to functions that actually belong to the generalized eigenspace associated to the small eigenvalues of $P_h$, let us now consider the operator 
\begin{align}\label{Pi0}
\Pi_0=\frac{1}{2i\pi}\int_{|z|=c h}(z-P_h)^{-1}\D z
\end{align}
introduced in \cite{Robbe}.
Using the resolvent estimates from Theorem \Ref{thmRobbe}, the following is established in \cite{Robbe}:

\begin{prop}\label{pi0}
The operator $\Pi_0$ is a projector on the generalized eigenspace associated to the small eigenvalues of $P_h$ and satisfies $\|\Pi_0\|=O(1)$.
\end{prop}

\begin{lem}\label{1-pi0}
Using the notations \eqref{tau}, \eqref{ftilde} and \eqref{lamtilde},
for any $\mathbf m\in \mathtt U^{(0)}$, we have 
$$\|(1-\Pi_0)\tilde f^{\mathbf m}_{\nu,h}\|=\sqrt{\tilde \lambda_{\mathbf m,h}} \bigg(O_\nu\big(1\big)+O\Big(h^{-1/2}\nu^{\frac{1}{2\sqrt{|\tau_{\mathbf s}|}}}\,|\ln (\nu)| \Big) \bigg).$$
\end{lem}

\begin{proof} 
We simply recall the proof from \cite{LPMichel}:
we write 
\begin{align*}
(1-\Pi_0)\tilde f^{\mathbf m}_{\nu,h}&=\frac{1}{2i\pi}\int_{|z|= c h}\big(z^{-1}-(z-P_h)^{-1}\big)\tilde f^{\mathbf m}_{\nu,h}\D z\\
		&=\frac{-1}{2i\pi}\int_{|z|= c h}z^{-1}(z-P_h)^{-1}P_h\tilde f^{\mathbf m}_{\nu,h}\D z.
\end{align*}
We can then conclude using Lemma \ref{Pf^2} and the resolvent estimate from Theorem \ref{thmRobbe}.
\end{proof}

\begin{lem}\label{Pi0fon}
Recall the notations \eqref{tau}, \eqref{ftilde} and \eqref{lamtilde}.
The family $\big(\Pi_0\tilde f^{\mathbf m}_{\nu,h}\big)_{\mathbf m\in \mathtt U^{(0)}}$ is almost orthonormal:
there exists $c>0$ such that 
\begin{align}\label{pi0fmpi0fm'}
\big\langle \Pi_0\tilde f^{\mathbf m}_{\nu,h}, \Pi_0\tilde f^{\mathbf m'}_{\nu,h}  \big\rangle=\delta_{\mathbf m, \mathbf m'}+O_\nu(\e^{-c/h}).
\end{align}
In particular, it is a basis of the space $\mathrm{Ran}\, \Pi_0$.\\
Moreover, we have
$$\big\langle P_h\Pi_0\tilde f^{\mathbf m}_{\nu,h}, \Pi_0\tilde f^{\mathbf m'}_{\nu,h}  \big\rangle=\delta_{\mathbf m, \mathbf m'}\tilde \lambda^{\mathbf m}_{\nu,h}+\sqrt{\tilde \lambda^{\mathbf m}_{\nu,h}\tilde \lambda^{\mathbf m'}_{\nu,h}} \,\bigg(O_\nu\Big(\sqrt h\,\Big)+O\Big(\nu^{\frac{1}{2\sqrt{|\tau_{\mathbf s}|}}}\,|\ln (\nu)|^2\Big)\bigg).$$
\end{lem}

\begin{proof}
The proof is the same as the one of Proposition 4.10 in \cite{LPMichel}.
It suffices to write 
$$\big\langle \Pi_0\tilde f^{\mathbf m}_{\nu,h}, \Pi_0\tilde f^{\mathbf m'}_{\nu,h}  \big\rangle=\big\langle \tilde f^{\mathbf m}_{\nu,h}, \tilde f^{\mathbf m'}_{\nu,h}  \big\rangle+\big\langle \tilde f^{\mathbf m}_{\nu,h}, (\Pi_0-1)\tilde f^{\mathbf m'}_{\nu,h}  \big\rangle+\big\langle (\Pi_0-1)\tilde f^{\mathbf m}_{\nu,h}, \Pi_0\tilde f^{\mathbf m'}_{\nu,h}  \big\rangle$$
as well as
$$
\big\langle P_h\Pi_0\tilde f^{\mathbf m}_{\nu,h}, \Pi_0\tilde f^{\mathbf m'}_{\nu,h}  \big\rangle=\big\langle P_h\tilde f^{\mathbf m}_{\nu,h}, \tilde f^{\mathbf m'}_{\nu,h}  \big\rangle+\big\langle (\Pi_0-1)\tilde f^{\mathbf m}_{\nu,h}, P_h^*\tilde f^{\mathbf m'}_{\nu,h}  \big\rangle+\big\langle \Pi_0P_h\tilde f^{\mathbf m}_{\nu,h}, (\Pi_0-1)\tilde f^{\mathbf m'}_{\nu,h}  \big\rangle.
$$
and use all the previous results of this section together with Proposition \ref{Pff}.
\end{proof}

\hop
Let us re-label the local minima $\mathbf m_1,\dots, \mathbf m_{n_0}$ so that $(S(\mathbf m_j))_{j=1,\dots,n_0}$ is non increasing in $j$.
For shortness, we will now denote 
$$\tilde f_j=\tilde f^{\mathbf m_j}_{\nu,h} \qquad \text{and} \qquad \tilde \lambda_j=\tilde \lambda^{\mathbf m_j}_{\nu,h}$$
which still depend on $\nu$ and $h$.
Note in particular that according to Proposition \ref{Pff}, $\tilde\lambda_j =O_\nu(\tilde\lambda_k)$ whenever $1\leq j\leq k \leq n_0$.
We also denote $(\tilde u_j)_{j=1,\dots,n_0}$ the orthogonalization by the Gram-Schmidt procedure of the family $(\Pi_0\tilde f_j)_{j=1,\dots,n_0}$ and 
$$u_j=\frac{\tilde u_j}{\|\tilde u_j\|}.$$
In this setting and with our previous results, we get the following (see \cite{LPMichel}, Proposition 4.12 for a proof).

\begin{lem}\label{Puu'}
With the notations \eqref{tau}, \eqref{ftilde} and \eqref{lamtilde},
for all $1\leq j,k \leq n_0$, it holds 
$$\langle P_h u_j, u_k \rangle=\delta_{j,k}\tilde \lambda_j +\sqrt{\tilde \lambda_j \tilde \lambda_k} \bigg(O_\nu\Big(\sqrt h\,\Big)+O\Big(\nu^{\frac{1}{2\sqrt{|\tau_{\mathbf s}|}}}\,|\ln (\nu)|^2\Big)\bigg).$$
\end{lem}


\hop
In order to compute the small eigenvalues of $P_h$, let us now consider the restriction $P_h|_{\mathrm{Ran}\,\Pi_0}:\mathrm{Ran}\,\Pi_0\to \mathrm{Ran}\,\Pi_0$.
We denote $\hat u_j=u_{n_0-j+1}$, $\hat \lambda_j=\tilde\lambda_{n_0-j+1}$ and $\mathcal M$ the matrix of $P_h|_{\mathrm{Ran}\,\Pi_0}$ in the orthonormal basis $(\hat u_1,\dots,\hat u_{n_0})$.
Since $\hat u_{n_0}=u_1=\tilde f_1$, we have 
$$\mathcal M=\begin{pmatrix}
		\mathcal M'&0\\
		0&0
\end{pmatrix} \qquad \text{where }\qquad \mathcal M'=\Big(\langle P_h\hat u_{j}, \hat u_k \rangle\Big)_{1\leq j,k\leq n_0-1}$$
and it is sufficient to study the spectrum of $\mathcal M'$.
We will also denote $\{\hat S_1< \dots < \hat S_p\}$ the set $\{S(\mathbf m_j)\, ; \, 2\leq j \leq n_0\}$ and for $1\leq k \leq p$, $E_k$ the subspace of $L^2(\R^{2d})$ generated by $\{\hat u_r\, ; \,  S(\mathbf m_r)=\hat S_k\}$.
Finally, we set $\varpi_k=\e^{-(\hat S_k-\hat S_{k-1})/h}$ for $2\leq k \leq p$ and $\varepsilon_j(\varpi)=\prod_{k=2}^j \varpi_k=\e^{-(\hat S_j-\hat S_1)/h}$ for $2\leq j \leq p$ (with the convention $\varepsilon_1(\varpi)=1$).
In view of Proposition \ref{Pff}, let us also denote
$$\tilde \varrho_0(\mathbf m)=\frac{1}{\pi}\sum_{\mathbf s\in \mathbf j(\mathbf m)}\bigg(\frac{2+\sqrt2}{2-\sqrt2}\bigg)^{\frac{1}{\sqrt{|\tau_{\mathbf s}|}}}\bigg(\frac{\det \mathcal V_{\mathbf m}}{|\det \mathcal V_{\mathbf s}|}\bigg)^{1/2}   \int_{\gamma_1\leq z\leq \gamma < 1} k_0^{\mathbf s}(\gamma) k_0^{\mathbf s}(z) \ln \bigg(2\,\frac{(1+z)(1+\gamma)}{1+3z+3\gamma+z\gamma}\bigg) \, \D z \D \gamma$$
and
$$\hat \lambda^0_{j}=h\, \tilde \varrho_{0}(\mathbf m_{n_0-j+1})\, \e^{\frac{-2S(\mathbf m_{n_0-j+1})}{h}}.$$

\begin{lem}\label{propgas}
With the above notations, the matrix $\mathcal M'$ satisfies
$$h^{-1}\e^{2\hat S_1/h}\mathcal M'=\Omega(\varpi) \bigg(M^\#_{0}+O_\nu\Big(\sqrt h\,\Big)+O\Big(\nu^{\frac{1}{2\sqrt{|\tau_{\mathbf s}|}}}\,|\ln (\nu)|^2\Big)\bigg) \Omega(\varpi)$$
with
$$M_0^\#=\mathrm{diag}\Big(\tilde \varrho_0(\mathbf m_{n_0-j+1})  \, ; \, 1\leq j \leq n_0-1\Big)$$
and
$$\Omega(\varpi)=\mathrm{diag}\big(\varepsilon_1(\varpi)\mathrm{Id}_{E_1},\dots,\varepsilon_p(\varpi)\mathrm{Id}_{E_p}\big).$$
In particular, for all $\nu>0$, there exists $h_0>0$ such that for all $0<h<h_0$,
$$h^{-1}\e^{2\hat S_1/h}\mathcal M'=\Omega(\varpi) \bigg(M^\#_{0}+O\Big(\nu^{\frac{1}{2\sqrt{|\tau_{\mathbf s}|}}}\,|\ln (\nu)|^2\Big)\bigg) \Omega(\varpi).$$
\end{lem}

\begin{rema}\label{remgas}
In the words of Definition A.1 from \cite{LPMichel}, the last Lemma implies that for all $\nu>0$, there exists $h_0>0$ such that for all $0<h<h_0$,
$$h^{-1}\e^{2\hat S_1/h}\mathcal M' \text{ is a } \Big((E_k)_k\,,\,\varpi \,,\, \nu^{\frac{1}{2\sqrt{|\tau_{\mathbf s}|}}}\,|\ln (\nu)|^2\Big)\text{-graded matrix}.$$
\end{rema}

\begin{proof}
According to Lemma \ref{Puu'} and Proposition \ref{Pff}, we can decompose $\mathcal M'=\mathcal M'_1+\mathcal M'_2$ with 
$$\mathcal M'_1=\mathrm{diag}(\hat\lambda_j^0 \, ; \, 1\leq j \leq n_0-1)\qquad \text{and} \qquad \mathcal M'_2=\bigg(\sqrt{\hat \lambda_j \hat \lambda_k} \bigg[O_\nu\Big(\sqrt h\,\Big)+O\Big(\nu^{\frac{1}{2\sqrt{|\tau_{\mathbf s}|}}}\,|\ln (\nu)|^2\Big)\bigg]\bigg)_{1\leq j,k\leq n_0-1}.$$
It then suffices to notice that $M^\#_0=h^{-1}\e^{2\hat S_1/h}\Omega(\varpi)^{-1}\mathcal M'_1 \Omega(\varpi)^{-1}$ and that
$$h^{-1}\e^{2\hat S_1/h}\Omega(\varpi)^{-1}\mathcal M'_2 \Omega(\varpi)^{-1}=O_\nu\Big(\sqrt h\,\Big)+O\Big(\nu^{\frac{1}{2\sqrt{|\tau_{\mathbf s}|}}}\,|\ln (\nu)|^2\Big)$$
where we still used Proposition \ref{Pff}.
\end{proof}
\hop
$\textit{Proof of Theorem \ref{thmToto}}$.
According to Remark \ref{remgas}, it now suffices to combine the result of Lemma \ref{propgas} with Theorem A.4 from \cite{LPMichel} which gives a description of the spectrum of graded matrices.
We get that for all $\nu>0$, there exists $h_0>0$ such that for all $0<h<h_0$, 
$$h^{-1}\e^{2S(\mathbf m)/h}\lambda(\mathbf m,h)-\tilde \varrho_0(\mathbf m)=O\Big(\nu^{\frac{1}{2\sqrt{|\tau_{\mathbf s}|}}}\,|\ln (\nu)|^2\Big)$$
and the result is proven.
\hspace*{\fill} $\Box$

\hop
$\textit{Proof of Corollaries \ref{ral} and \ref{meta}}$.
With the notations from Theorem \ref{thmToto},  it is shown in \cite{Robbe}, section 4 with the use of PT-Symmetry arguments and a quantitative version of the Gearhart-Pr\"uss Theorem, that there exist $c>0$ and some projectors $(\Pi_j)_{1\leq j \leq n_0}$ which are $O(1)$ and such that
\begin{enumerate}[label=\textbullet]
\item $\Pi_1=\Pro_1$
\item $\Pi_j\Pi_k=\delta_{j,k}\Pi_j$
\item $\Pro_k=\sum_{\{j\, ;\,S(\mathbf m_j)\geq S(\mathbf m_k)\}}\Pi_j$
\item $\e^{-tP_h/h}=\sum_{j=1}^{n_0} \e^{-t\lambda(\mathbf m_j,h)/h}\Pi_{j}+O(\e^{-ct})$ \quad for $t\geq0$ and $h$ small enough.
\end{enumerate}
Corollary \ref{ral} directly follows, while the proof of Corollary \ref{meta} is then an easy adaptation of the one of Corollary 1.6 from \cite{BonyLPMichel}.
(Note that our notations $t_k^-$ and $t_k^+$ differ from that in \cite{BonyLPMichel}).
\hspace*{\fill} $\Box$

\addtocontents{toc}{\SkipTocEntry}
\subsection*{Acknowledgements}
The author is grateful to Jean-François Bony for very helpful discussions as well as Laurent Michel for his advice through this work, and especially for the proof of Proposition \ref{facto}.\\
This work is supported by the ANR project QuAMProcs 19-CE40-0010-01.

\appendix

\section{Structure of the collision operator}\label{qmicro}

The aim of this section is to show Proposition \ref{facto} and Corollary \ref{opgb}.
For $a$, $b$ two symbols, we denote $a\#b$ the symbol of $\mathrm{Op}_h(a) \circ \mathrm{Op}_h(b)$.
We start by showing that $Q_h$ defined in \eqref{qlr} is a pseudo-differential operator:

\begin{lem}\label{piop}
One has $\Pi_h=\mathrm{Op}_h(\varpi_h)$ with $\varpi_h\in S^{1/2}(1)$ given by
$$\varpi_h(v,\eta)=2^d\e^{-\frac{v^2+4\eta^2}{2h}}.$$ 
\end{lem}
\hip
\textit{Proof of Lemma \ref{piop}.}
First, notice that the distributional kernel of $\Pi_h$ is $\mu_h(v)\mu_h(v')$.
Using the formula \eqref{FK} to compute the symbol of a pseudo-differential operator from its distributional kernel, we get
$$\mathcal F_{h,v'}\Big(\mu_h(v+v'/2)\mu_h(v-v'/2)\Big)(v,\eta)=2^d\e^{-\frac{v^2+4\eta^2}{2h}}$$
which is clearly in $S^{1/2}(1)$ as $\e^{-\frac{v^2+4\eta^2}{2}}\in S^0(1)$. \hfill $\Box$

\hop
\textit{Proof of Proposition \ref{facto}.}
Let us first check that $m_h\in S^{1/2}\big(\langle v, \eta \rangle^{-2}\big)$.
We have 
\begin{align}\label{tildecheck}
m_h(v,\eta)=\tilde m(h^{-1/2}v,h^{-1/2}\eta)\qquad \text{and}\qquad \tilde m(v,\eta)=\check m\Big(\frac{v^2}{2}+2\eta^2\Big)
\end{align}
with
$$\tilde m(v,\eta)=2\int_0^1 (y+1)^{d-2} \e^{-y\big(\frac{v^2}{2}+2\eta^2\big)} \D y\qquad \text{and}\qquad \check m(\mu)=2\int_0^1(y+1)^{d-2}\e^{-y \mu}\D y.$$
One can then check using integration by parts that for all $k\in \N$, there exists $C_k$ such that $|\partial^k_\mu \check m(\mu)|\leq C_k\langle\mu\rangle^{-k-1}$ from which we deduce using \eqref{tildecheck} that $\tilde m\in S^0\big(\langle v, \eta \rangle^{-2}\big)$.
Thus, still using \eqref{tildecheck}, for $\alpha\in \N^{2d}$, there exists $C_\alpha$ such that 
$$|\partial^\alpha m_h(v,\eta)|=h^{-|\alpha|/2}\big|\partial^\alpha \tilde m(h^{-1/2}v,h^{-1/2}\eta)\big|\leq C_\alpha h^{-|\alpha|/2}\big\langle h^{-1/2}v, h^{-1/2}\eta \big\rangle^{-2}\leq C_\alpha h^{-|\alpha|/2}\langle v, \eta \rangle^{-2},$$
so $m_h$ indeed belongs to $S^{1/2}\big(\langle v, \eta \rangle^{-2}\big)$.
Using symbolic calculus and Lemma \ref{piop}, one could then simply check that 
\begin{align}\label{sharpsharp}
(-i\eta^t+v^t/2)\#(m_h\, \mathrm{Id})\#(i\eta+v/2)=h(1-\varpi_h)
\end{align}
but let us explain how the suitable $m_h$ (i.e the one solving \eqref{sharpsharp}) was found.
Since $(i\eta+v/2)$ and its conjugate are both polynomials of degree 1, we compute
\begin{align}\label{sharpsharp2}
(-i\eta^t+v^t/2)\#(m_h\, \mathrm{Id})\#(i\eta+v/2)=\Big(\eta^2+&\frac{v^2}{4}\Big)m_h\\
		-\frac h2 \big(dm_h+v\cdot\partial_vm_h&+\eta\cdot\partial_\eta m_h\big)+\frac{h^2}{4}\Big(\Delta_v+\frac14 \Delta_\eta\Big)m_h.
\end{align}
Let us look for solutions under the form $m_h(v,\eta)=u_h(v,\eta)\e^{\frac{v^2+4\eta^2}{2h}}$. In that case,
$$\partial_v m_h=\e^{\frac{v^2+4\eta^2}{2h}}\Big(\partial_v u_h+\frac {u_h}{h} v\Big)\qquad \text{and} \qquad \Delta_vm_h=\Big( \Delta_v u_h+\frac{2v}{h}\cdot \partial_v u_h+\frac dh u_h+\frac{v^2}{h^2}u_h\Big)  \e^{\frac{v^2+4\eta^2}{2h}} $$
so
$$\frac{h^2}{4}\Delta_vm_h -\frac h2 v\cdot\partial_vm_h=\Big( \frac{h^2}{4}\Delta_v u_h+\frac {hd}{4} u_h-\frac{v^2}{4}u_h\Big)  \e^{\frac{v^2+4\eta^2}{2h}}.$$
Similarly, we compute
$$\frac{h^2}{16}\Delta_\eta m_h -\frac h2 \eta\cdot\partial_\eta m_h=\Big( \frac{h^2}{16}\Delta_\eta u_h+\frac {hd}{4} u_h-\eta^2 u_h\Big)  \e^{\frac{v^2+4\eta^2}{2h}}$$
so according to \eqref{sharpsharp2}, \eqref{sharpsharp} becomes
$$\frac{h^2}{4}\Big(\Delta_v u_h+\frac 14 \Delta_\eta u_h\Big)=h\Big(\e^{-\frac{v^2+4\eta^2}{2h}}-2^d\e^{-\frac{v^2+4\eta^2}{h}}\Big).$$
Applying the semiclassical Fourier transform on $\R^{2d}$, this yields
\begin{align}
-\frac 14 \Big({v^*}^2+\frac{{\eta^*}^2}{4}\Big)\mathcal F_h u_h
=h(\pi h)^d  \Big(\e^{-\frac{4{v^*}^2+{\eta^*}^2}{8h}}-\e^{-\frac{4{v^*}^2+{\eta^*}^2}{16h}}\Big)=-\frac{(\pi h)^d}{4}\Big({v^*}^2+\frac{{\eta^*}^2}{4}\Big)\int_1^2 \e^{-s\frac{4{v^*}^2+{\eta^*}^2}{16h}}\D s
\end{align}
where $(v^*,\eta^*)$ denotes the dual variable of $(v,\eta)$.
Hence 
$$\mathcal F_h u_h(v^*,\eta^*)=(\pi h)^d\int_1^2 \e^{-s\frac{4{v^*}^2+{\eta^*}^2}{16h}}\D s$$
and applying the inverse semiclassical Fourier transform, we get 
$$u_h(v,\eta)=2^d\int_1^2 s^{-d}\e^{-\frac{v^2+4\eta^2}{sh}}\D s.$$
Consequently,
$$m_h(v,\eta)= 2^d\int_1^2 s^{-d}\e^{-\frac{v^2+4\eta^2}{2h}(\frac 2s -1)}\D s$$
and we find the final expression of $m_h$ by substituting $y=\frac 2s -1$. 
\hfill $\Box$

\hop
\textit{Proof of Corollary \ref{opgb}.}
By symbolic calculus, we just have to check that $g_h=(-i\eta^t+v^t/2)\# (m_h\, \mathrm{Id})$.
Since the symbol on the left hand side is a polynomial of degree 1, we have 
$$(-i\eta^t+v^t/2)\# (m_h\, \mathrm{Id})=m_h(-i\eta^t+ v^t/2)-\frac{h}{2}\Big( \partial_v^t-\frac i2  \partial_\eta^t\Big)m_h.$$
Now 
$$-\frac h2 \partial_v^t m_h(v,\eta)=\int_0^1y(y+1)^{d-2} \e^{-\frac{y}{h}\big(\frac{v^2}{2}+2\eta^2\big)} \D y\, v^t$$
so we easily get
$$m_h(v,\eta) \,\frac{v^t}{2}-\frac{h}{2}\partial_v^t m_h(v,\eta)=\int_0^1 (y+1)^{d-1} \e^{-\frac{y}{h}\big(\frac{v^2}{2}+2\eta^2\big)} \D y\,v^t.$$
One can show similarly that 
$$-im_h(v,\eta)\,\eta^t+\frac{ih}{4}\partial_\eta^t m_h(v,\eta)=-2i\int_0^1 (y+1)^{d-1} \e^{-\frac{y}{h}\big(\frac{v^2}{2}+2\eta^2\big)} \D y\,\eta^t$$
which is enough to conclude. \hfill $\Box$

\section{Bilinear algebra}

\begin{lem}\label{det}
Let $L(x,v)=L_x\cdot x+L_v\cdot v$ a linear form on $\R^{2d}$ and recall the notation \eqref{hessiennes}.
Then for any $\mathbf s\in \mathtt U^{(1)}$, the matrix $\mathcal W_{\mathbf s} +\nabla L\, \nabla L^t$ is positive definite if and only if 
\begin{align}\label{w+l2pos}
-\mathcal V_{\mathbf s}^{-1}L_x\cdot L_x-L_v^2>\frac12.
\end{align}
Moreover, its determinant is
$$2^{-2d} \det \mathcal V_{\mathbf s}\,\big(1+2\mathcal V_{\mathbf s}^{-1}L_x\cdot L_x+2L_v^2\big).$$
\end{lem}

\begin{proof}
First notice that since $\mathbf s\in \mathtt U^{(1)}$ and $\mathcal W_{\mathbf s} +\nabla L\, \nabla L^t \geq \mathcal W_{\mathbf s}$,
the matrix $\mathcal W_{\mathbf s} +\nabla L\, \nabla L^t$ has at most one negative eigenvalue, so it is sufficient to show that its determinant is positive if and only if \eqref{w+l2pos} holds.
The rest of the proof is inspired by \cite{BonyLPMichel} (Lemma 3.3).
We have 
\begin{align}\label{detdet}
\det\Big(\mathcal W_{\mathbf s} +\nabla L\, \nabla L^t\Big)=\det \mathcal W_{\mathbf s}\det\Big(\mathrm{Id}+\mathcal W_{\mathbf s}^{-1}\nabla L\, \nabla L^t\Big)=2^{-2d} \det \mathcal V_{\mathbf s}\det\Big(\mathrm{Id}+\mathcal W_{\mathbf s}^{-1}\nabla L\, \nabla L^t\Big)
\end{align}
and since $\det \mathcal V_{\mathbf s}<0$, it only remains to show that
\begin{align}\label{posiff}
\eqref{w+l2pos}\iff \det\Big(\mathrm{Id}+\mathcal W_{\mathbf s}^{-1}\nabla L\, \nabla L^t\Big)<0.
\end{align}
Now it is easy to see that
$$\Big(\mathrm{Id}+\mathcal W_{\mathbf s}^{-1}\nabla L\, \nabla L^t\Big)|_{\nabla L^\perp}=\mathrm{Id}\quad \; \text{and} \quad \;  \Big(\mathrm{Id}+\mathcal W_{\mathbf s}^{-1}\nabla L\, \nabla L^t\Big)\nabla L\cdot \nabla L=\big(1+2\mathcal V_{\mathbf s}^{-1}L_x\cdot L_x+2L_v^2\big)|\nabla L|^2.$$
Hence, $\det\big(\mathrm{Id}+\mathcal W_{\mathbf s}^{-1}\nabla L\, \nabla L^t\big)=1+2\mathcal V_{\mathbf s}^{-1}L_x\cdot L_x+2L_v^2$ which is negative if and only if \eqref{w+l2pos} holds true.
\end{proof}

\begin{lem}\label{det2}
Recall the notations \eqref{hessiennes} and \eqref{hy>}. For $\gamma \in [\gamma_1+\nu,1]$ and $y\in (0,1)$, we have 
\begin{align}\label{dethy}
{}\qquad \det H^{\mathbf s}_{\gamma,y}=\frac{(1+y)^{2d-2}}{(4y)^d} \Big(1+\big(1+2|L_{\gamma,v}^{\mathbf s}|^2\big)y\Big)^2\,|\det \mathcal V|.
\end{align}
\end{lem}

\begin{proof}
We drop some exponents and indexes $\mathbf s$ in the notations for shortness.
Let us begin by writing
\begin{align}\label{1+htilde}
\text{ }\qquad H_{\gamma,y}
&=
\begin{pmatrix}  \mathcal V&0&0 \\
0 & \frac{(y+1)^2}{4y}  & \frac{y^2-1}{4y} \\
0&  \frac{y^2-1}{4y} & \frac{(y+1)^2}{4y} \end{pmatrix}
\left[
\mathrm{Id}+
\begin{pmatrix}
\mathcal V^{-1}&0&0 \\
0 & 1  & \frac{1-y}{1+y} \\
0&  \frac{1-y}{1+y} & 1
\end{pmatrix}
\begin{pmatrix}
L_{\gamma,x} & L_{\gamma,x}\\
L_{\gamma,v}&0\\
0&L_{\gamma,v}
\end{pmatrix}
\begin{pmatrix}
L_{\gamma,x}^t & L_{\gamma,v}^t & 0\\
L_{\gamma,x}^t&0&L_{\gamma,v}^t
\end{pmatrix}
\right].
\end{align}
Clearly, the determinant of the first factor is $(4y)^{-d}(y+1)^{2d}\det \mathcal V$.
Denoting
$$\tilde H_{\gamma,y}=\begin{pmatrix}
\mathcal V^{-1}&0&0 \\
0 & 1  & \frac{1-y}{1+y} \\
0&  \frac{1-y}{1+y} & 1
\end{pmatrix}
\begin{pmatrix}
L_{\gamma,x} & L_{\gamma,x}\\
L_{\gamma,v}&0\\
0&L_{\gamma,v}
\end{pmatrix}
\begin{pmatrix}
L_{\gamma,x}^t & L_{\gamma,v}^t & 0\\
L_{\gamma,x}^t&0&L_{\gamma,v}^t
\end{pmatrix},$$
it is also clear that $\tilde H_{\gamma,y}$ has rank 2, so it has at most 2 non zero eigenvalues.
Besides, using Lemma \ref{bonmin}, one can easily check that
$$\tilde H_{\gamma,y}\begin{pmatrix}(1+y)\mathcal V^{-1}L_{\gamma,x}\\ L_{\gamma,v}\\ L_{\gamma,v}\end{pmatrix}=\frac{-2}{1+y}\Big(1+\big(1+|L_{\gamma,v}^{\mathbf s}|^2\big)y\Big)\begin{pmatrix}(1+y)\mathcal V^{-1}L_{\gamma,x}\\ L_{\gamma,v}\\ L_{\gamma,v}\end{pmatrix}$$
and
$$\tilde H_{\gamma,y}\begin{pmatrix}0\\ L_{\gamma,v}\\ -L_{\gamma,v}\end{pmatrix}=\frac{2y|L_{\gamma,v}|^2}{1+y}\begin{pmatrix}0\\ L_{\gamma,v}\\ -L_{\gamma,v}\end{pmatrix}.$$
Hence, 
the determinant of the second factor from \eqref{1+htilde} is
$$-(1+y)^{-2}\Big(1+\big(1+2|L_{\gamma,v}^{\mathbf s}|^2\big)y\Big)^2$$
and we get \eqref{dethy}.
\end{proof}

\section{Multivariate gaussian moment}
\hip
Using the formulas of the first moments of the one dimensional gaussian, we easily establish the following.
\begin{prop}\label{mom2}
If $A$ is a real symmetric matrix, then 
$$\int_{\R^{d'}}Ax\cdot x\, \e^{-\frac{x^2}{2}}\D x=(2\pi)^{d'/2}\,\mathrm{Tr}(A).$$
\end{prop}

\section{Laplace's method}
\hip
Here we give a precise statement of Laplace's method that we use to approximate $h$-dependent integrals.

\begin{prop}\label{lap}
Let $x_0 \in  \R^{d'}$, $K$ a compact neighborhood of $x_0$ and $\varphi\in \mathcal C^\infty(K)$ such that $x_0$ is a non degenerate minimum of $\varphi$ and its only global minimum on $K$.
Denote $H\in \mathcal M_{d'}(\R)$ the Hessian of $\varphi$ at $x_0$.
\begin{enumerate}[label=\textbullet]
\item If $a_h$ is a function bounded uniformly in $h$ on $K$ such that 
$$a_h=O\Big((x-x_0)^{2n}\Big),$$
then
$$h^{-d'/2}\int_{K}a_h(x)\e^{-\frac{\varphi(x)-\varphi(x_0)}{h}}\D x=O(h^n).$$
\item If $a_h \sim \sum_{j\geq 0}h^j a_j $ in $\mathcal C^\infty (K)$,
then the integral
$$\frac{\det (H)^{1/2}}{(2\pi h)^{d'/2}} \int_{K}a_h(x)\e^{-\frac{\varphi(x)-\varphi(x_0)}{h}}\D x$$
admits a classical expansion whose first term is given by $a_0(x_0)$.
\end{enumerate}
\end{prop}

\nocite{}
\bibliography{bibBoltz} 
\bigskip
\scshape \small Thomas Normand, Laboratoire de Math\'ematiques Jean Leray, Universit\'e de Nantes
\normalfont

\end{document}